\documentclass[10pt]{amsart}
\usepackage{amsrefs}
\usepackage{amssymb}
\usepackage[all,arc]{xy}
\usepackage{graphicx}
\usepackage{mathrsfs}
\usepackage{xcolor}
\usepackage{hyperref}

\usepackage{tikz}
\usetikzlibrary{matrix,arrows}

\newcommand{\Aut}{\mathrm{Aut}}

\newcommand{\cd}{\mathrm{cd}}

\newcommand{\Gal}{\mathrm{Gal}}

\newcommand{\image}{\mathrm{im}}
\newcommand{\Inf}{\mathrm{inf}}

\newcommand{\kernel}{\mathrm{ker}}

\newcommand{\Res}{\mathrm{res}}

\newcommand{\chr}{\mathrm{char}}
\newcommand{\Or}{\mathbf{Or}}
\newcommand{\ob}{\mathrm{ob}}

\newcommand{\euV}{\mathscr{V}}
\newcommand{\euE}{\mathscr{E}}

\newcommand{\eue}{\mathbf{e}}
\newcommand{\K}{\mathbb{K}}
\newcommand{\LL}{\mathbb{L}}

\newcommand{\N}{\mathbb{N}}

\DeclareFontFamily{U}{wncy}{}
\DeclareFontShape{U}{wncy}{m}{n}{<->wncyr10}{}
\DeclareSymbolFont{mcy}{U}{wncy}{m}{n}
\DeclareMathSymbol{\Sha}{\mathord}{mcy}{"58}
\DeclareMathSymbol{\sha}{\mathord}{mcy}{"78}

\begin{document}

\newtheorem{thm}{Theorem}[section]
\newtheorem{cor}[thm]{Corollary}
\newtheorem{lem}[thm]{Lemma}
\newtheorem{prop}[thm]{Proposition}
\newtheorem{defin}[thm]{Definition}
\newtheorem{exam}[thm]{Example}
\newtheorem{rem}[thm]{Remark}
\newtheorem{case}{\sl Case}
\newtheorem{claim}{Claim}
\newtheorem{fact}[thm]{Fact}
\newtheorem{question}[thm]{Question}
\newtheorem{conj}[thm]{Conjecture}
\newtheorem*{notation}{Notation}
\swapnumbers
\newtheorem{rems}[thm]{Remarks}
\newtheorem*{acknowledgment}{Acknowledgment}

\newtheorem{questions}[thm]{Questions}
\numberwithin{equation}{section}

\newcommand{\gr}{\mathrm{gr}}
\newcommand{\inv}{^{-1}}
\newcommand{\isom}{\cong}
\newcommand{\dbC}{\mathbb{C}}
\newcommand{\F}{\mathbb{F}}
\newcommand{\dbN}{\mathbb{N}}
\newcommand{\Q}{\mathbb{Q}}
\newcommand{\dbR}{\mathbb{R}}
\newcommand{\dbU}{\mathbb{U}}
\newcommand{\Z}{\mathbb{Z}}
\newcommand{\calG}{\mathcal{G}}
\newcommand{\sep}{\mathrm{sep}}
\newcommand{\ab}{\mathrm{ab}}
\newcommand{\Zen}{\mathrm{Z}}
\newcommand{\baK}{\bar{\K}^{\sep}}
\newcommand{\triv}{\{1\}}
\newcommand{\batheta}{\bar{\theta}}
\newcommand{\ttheta}{\tilde{\theta}}
\newcommand{\bone}{\mathbf{1}}
\newcommand{\cl}{\mathrm{cl}}
\newcommand{\Iso}{\mathrm{Iso}}
\newcommand{\ET}{\mathbf{ET}}
\newcommand{\argu}{\hbox to 7truept{\hrulefill}}
\newcommand{\bos}{\mathbf{s}}
\newcommand{\Ext}{\mathrm{Ext}}
\newcommand{\dbl}{[\![}
\newcommand{\dbr}{]\!]}
\newcommand{\ZLG}{\Z_\ell\dbl G\dbr}
\newcommand{\Isl}{\mathrm{Iso}}
\newcommand{\baG}{\bar{G}}
\newcommand{\Lie}{\mathscr{L}}
\newcommand{\tor}{\mathrm{tor}}
\newcommand{\boT}{\mathbf{T}}
\newcommand{\boA}{\mathbf{A}}
\newcommand{\boS}{\mathbf{S}}
\newcommand{\Pb}{\mathrm{Pb}}
\newcommand{\bfLam}{\mathbf{\Lambda}}
\newcommand{\bfXi}{\mathbf{\Xi}}


\newcommand{\hac}{\hat c}
\newcommand{\hatheta}{\hat\theta}

\title[Oriented pro-$\ell$ groups with the Bogomolov-Positselski property]{Oriented pro-$\ell$ groups  with the 
Bogomolov-Positselski property}
\author{Claudio Quadrelli}
\author{Thomas S. Weigel}
\address{Department of Mathematics and Applications, University of Milano Bicocca, 20125 Milan, Italy, EU}
\email{claudio.quadrelli@unimib.it}
\address{Department of Mathematics and Applications, University of Milano Bicocca, 20125 Milan, Italy, EU}
\email{thomas.weigel@unimib.it}
\date{\today}

\begin{abstract}
For a prime number $\ell$
we say that an oriented pro-$\ell$ group $(G,\theta)$ has the Bogomolov-Positselski property if 
the kernel of the
canonical projection on its maximal $\theta$-abelian quotient $\pi^{\ab}_{G,\theta}\colon G\to G(\theta)$ is a free 
pro-$\ell$ group contained in the Frattini subgroup of $G$.
We show that oriented pro-$\ell$ groups of elementary type have the Bogomolov-Positselski property (cf. Theorem~\ref{thm:one}).
This shows that Efrat's Elementary Type Conjecture implies a positive answer to Positselski's version of Bogomolov's Conjecture on maximal pro-$\ell$ Galois groups of a field $\K$ in case that $\K^\times/(\K^\times)^\ell$ is finite.
Secondly, it is shown that for an $H^\bullet$-quadratic oriented pro-$\ell$ group $(G,\theta)$ the Bogomolov-Positselski property can be expressed by the injectivity of the transgression map $d_2^{2,1}$ in the Hochschild-Serre spectral sequence (cf. Theorem~\ref{thm:two}).
\end{abstract}

\subjclass[2010]{Primary 12G05; Secondary 20E18, 20J06, 12F10}

\keywords{Maximal pro-$\ell$ Galois groups, Bogomolov's Conjecture, oriented pro-$\ell$ groups, 
Kummerian oriented pro-$\ell$ groups}

\maketitle

\section{Introduction}
\label{sec:intro}
By an {\sl $\ell$-oriented profinite group} for a prime number $\ell$ we understand a profinite group $G$ together
with a continuous homomorphism of profinite groups $\theta\colon G\to\Z_\ell^\times$, where $\Z_\ell^\times$
denotes the group of units of the ring of $\ell$-adic integers $\Z_\ell$.
An $\ell$-oriented pro-$\ell$ group $(G,\theta)$ will be simply called an {\sl oriented pro-$\ell$ group}.
For a field $\K$, we denote by $G_{\K}=\Gal(\baK/\K)$ its absolute Galois group, where $\baK$ denotes a separable closure of $\K$. For any prime number $\ell$, $G_{\K}$ carries naturally the cyclotomic $\ell$-orientation
$\tilde\theta_{\K,\ell}\colon G_{\K}\to \Z_\ell^\times$ (cf. Example~\ref{ex:Galois} and \cite[(1.3)]{qw:cyclotomic}).
The following conjecture formulated by L.~Positselski in \cite[Conjecture~2]{pos:k} was motivated by an
earlier conjecture of F.~Bogomolov (cf. \cite{bogo} and \cite[Conjecture~1]{pos:k}, see also Remark~\ref{rem:history Bogoconj} below).

\begin{conj}\label{conj:bogo_intro}
 Let $\K$ be a field containing a primitive $\ell^{\mathrm{th}}$-root of unity,
and also $\sqrt{-1}$ if $\ell=2$, and set $$\sqrt[\ell^\infty]{\K}=\K\left(\sqrt[\ell^n]{a},a\in\K,n\geq1\right).$$
Then the maximal pro-$\ell$ Galois group of $\sqrt[\ell^\infty]{\K}$ is a free pro-$\ell$ group.
\end{conj}

A profinite group $G$ admits a maximal pro-$\ell$ quotient $G(\ell)=G/O^\ell(G)$, where
$O^\ell(G)$ is the closed normal subgroup of $G$ being generated by all pro-$q$ Sylow  subgroups for all prime numbers $q\not=\ell$. 
Apart from $\kernel(\theta)$, an oriented pro-$\ell$ group $(G,\theta)$ contains the distinguished closed subgroups
\begin{align}
K_\theta(G)&=\cl\left(\left\langle\: h^{-\theta(g)}ghg^{-1}\:\mid\: g\in G,h\in\kernel(\theta)
\:\right\rangle\right)\label{eq:defK2}\\
\intertext{and}
I_\theta(G)&=\cl\left(\left\langle\:h\in\kernel(\theta)\:\mid\: \exists k\in\dbN_0: h^{\ell^k}\in K_\theta(G)
\:\right\rangle\right)\label{eq:defItheta}
\end{align}
--- the former introduced in \cite{eq:kummer} --- i.e., $I_\theta(G)$ is the closure of the isolator (cf. \cite[\S 66]{kur:grp}) of $K_\theta(G)$ in $\kernel(\theta)$.
An oriented pro-$\ell$ group $(G,\theta)$ is said to be {\sl $\theta$-abelian}, if 
{the subgroup $K_{\theta}(G)$ is trivial and if
$\ker(\theta)$ is a free abelian pro-$\ell$ group} (in this case $G$ is a free abelian-by-cyclic pro-{$\ell$} group 
{for $\ell\not=2$}, cf. Remark~\ref{rem:thetabel}).
By definition $K_\theta(G)$ is a closed normal subgroup of $G$ contained in the Frattini subgroup 
$\Phi(G)=\cl(G^\ell\cdot[G,G])$
of $G$. Note that
\begin{equation}
\label{eq:defK3}
[\kernel(\theta),\kernel(\theta)]\subseteq K_\theta(G)\subseteq\kernel(\theta),
\end{equation}
so that the quotient $\kernel(\theta)/K_\theta(G)$ is an abelian pro-$\ell$ group, and $I_{\theta}(G)/K_\theta(G)$ is its torsion subgroup.
In particular, if $\theta\colon G\to\Z_\ell^\times$ is trivial (i.e., $\theta$ is identically equal to 1), then 
$K_\theta(G)$ coincides with the closure of the commutator subgroup of $G$.

Every oriented pro-$\ell$ group $(G,\theta)$ admits a maximal $\theta$-abelian quotient
$(G(\theta),\batheta)$, where $G(\theta)=G/I_\theta(G)$ and $\batheta\colon G(\theta)\to\Z_\ell^\times$ is the homomorphism induced by $\theta$.
Namely, $(G(\theta),\batheta)$ is $\batheta$-abelian and one has a canonical surjective homomorphism 
$$\pi^{\ab}_{G,\theta}\colon (G,\theta)\longrightarrow (G(\theta),\batheta)$$ of oriented pro-$\ell$ groups satisfying the following: for every homomorphism $\psi\colon (G,\theta)\to (A,\theta^\circ)$
of oriented pro-$\ell$ groups onto a $\theta^\circ$-abelian pro-$\ell$ group $(A,\theta^\circ)$ there exists a unique
homomorphism of oriented pro-$\ell$ groups $\psi^{\ab}_\theta\colon (G(\theta),\batheta)\to (A,\theta^\circ)$
such that $\psi=\psi^{\ab}_\theta\circ\pi^{\ab}_{G,\theta}$ (cf. Proposition~\ref{prop:maxTh}).

The hypothesis of Conjecture~\ref{conj:bogo_intro} on the primitive $\ell^{\mathrm{th}}$-roots lying in $\K$ implies that the maximal pro-$\ell$ quotient 
$G_{\K}(\ell)$
of the absolute Galois group $G_{\K}$ carries naturally an $\ell$-orientation 
\begin{equation}
\label{eq:cycchar}
\ttheta_{\K,\ell}\colon G_{\K}(\ell)\longrightarrow\Z_\ell^\times.
\end{equation} 
So, Conjecture~\ref{conj:bogo_intro} predicts that
$I_{\ttheta_{\K,\ell}}(G_\K(\ell))$
is a free pro-$\ell$ group contained in the Frattini subgroup $\Phi(G_\K(\ell))$ of $G_\K(\ell)$ (cf. Proposition~\ref{prop:kummer} and \S~\ref{ssec:bogo}).
At this point it should be mentioned that in fact one has to deal with two properties of oriented pro-$\ell$ groups.
The oriented pro-$\ell$ group $(G,\theta)$ is said to be {\sl Kummerian}, if $I_{\theta}(G)=\kernel(\pi^{\ab}_{G,\theta})$
is contained in the Frattini subgroup $\Phi(G)$ of $G$. This property can be reformulated in several different ways
(cf. Proposition~\ref{prop:kummer}).
Bearing this fact in mind we say that the Kummerian (cf. \S \ref{ss:kummer}) oriented pro-$\ell$ group $(G,\theta)$ has the {\sl Bogomolov-Positselski property}, if
$I_\theta(G)=\kernel(\pi^{\ab}_{G,\theta})$ is a free pro-$\ell$ group. 
E.g., the oriented pro-$\ell$ group $(G,\bone)$, where $\bone$ is the trivial $\ell$-orientation on $G$, is Kummerian if, and only if, the maximal abelian pro-$\ell$ quotient $G^{\ab}=G/G^\prime$ is a free abelian pro-$\ell$ group, and has the Bogomolov-Positselski property if, and only if, it is Kummerian and the closure of the commutator subgroup of $G$ is a free pro-$\ell$ group.

The class of oriented pro-$\ell$ groups $\ET_\ell$ of {\sl elementary type}  is the smallest class of oriented pro-$\ell$ groups
containing $\Z_\ell$ with all its $\ell$-orientations, all Demushkin pro-$\ell$ groups with their natural $\ell$-orientation
(cf. \cite[Proposition~5.2]{qw:cyclotomic}) and which is closed with respect to free products 
in the category of oriented pro-$\ell$ groups and fibre
products (cf. \S~\ref{ssec:ETC}). The Elementary Type Conjecture formulated by Ido Efrat in \cite{efrat:small} predicts that for every field $\K$ containing an $\ell^{\mathrm{th}}$-root of unity (and also $\sqrt{-1}$ if $\ell=2$) satisfying
$|\K^\times/(\K^\times)^\ell|<\infty$ the oriented pro-$\ell$ group $(G_\K(\ell),\ttheta_{\K,\ell})$  must be of elementary type. 
The first main purpose of this paper is to establish the following theorem relating the Elementary Type Conjecture
with Conjecture~\ref{conj:bogo_intro}.

\begin{thm}\label{thm:one}
Every oriented pro-$\ell$ group of elementary type has the Bogomolov-Posit\-selski property.
\end{thm}

From Theorem~\ref{thm:one} one concludes the following (cf. Proposition~\ref{prop:fields ET}):

\begin{cor}\label{cor:fields}
 Let $\K$ be a field containing a primitive $\ell^{\mathrm{th}}$-root of 1 {\rm (}and also $\sqrt{-1}$ if $\ell=2${\rm)}, such that the quotient $\K^\times/(\K^\times)^\ell$ is finite. 
 Then Conjecture~\ref{conj:bogo_intro} holds true in the following cases:
 \begin{itemize}
  \item[(a)] $\K$ is finite;
\item[(b)] $\K$ is a pseudo algebraically closed (PAC) field, or an extension of relative trascendence degree 1 of a PAC field;
\item[(c)] $\K$ is an extension of trascendence degree 1 of a local field;
\item[(d)] $\K$ is $\ell$-rigid {\rm (}for the definition of $\ell$-rigid field see \cite[p.~722]{ware}{\rm )};
\item[(e)] $\K$ is an algebraic extension of a global field of characteristic not $\ell$.
 \end{itemize}
\end{cor}

By the Norm Residue Theorem (cf. \cite{HW:book,voev,weibel,weibel2}), the mod $\ell$-Milnor $K$-ring $K_\bullet^M(\K)/\ell$ of a field $\K$ is isomorphic to the cohomology algebra $H^\bullet(G_{\K}(\ell),\F_\ell)$ provided $\ell\not=\chr(\K)$
and $\K$ contains a primitive $\ell^{\mathrm{th}}$-root of unity.
Moreover, L.~Positselski showed in \cite[Theorem~1.4]{pos:k} that Conjecture~\ref{conj:bogo_intro}
is a consequence of a strong Koszulity property of the cohomology algebra $H^\bullet(G_{\K}(\ell),\F_\ell)$.

Our second objective is to establish the following criterion ensuring the Bogomolov-Positselski property
of an abstract oriented pro-$\ell$ group $(G,\theta)$. Surprisingly, it only depends on low-dimensional
group cohomology, but in a sophisticated way (cf. Theorem~\ref{prop:bogo2}).

\begin{thm}
\label{thm:two}
Let $(G,\theta)$ be a Kummerian oriented pro-$\ell$ group with a quadratic $\F_\ell$-cohomo\-logy algebra $H^\bullet(G,\F_\ell)$, and let
\begin{equation}
\label{eq:seq}
\bos\colon\qquad\xymatrix{{\triv}\ar[r]&I_\theta(G)\ar[r]&G\ar[r]&G(\theta)\ar[r]&{\triv}}
\end{equation}
be the canonical extension of pro-$\ell$ groups. 
Then $G$ has the Bogomolov-Positselski property if, and only if, the transgression map
\begin{equation}
\label{eq:trans1}
 d_2^{2,1}\colon H^2(G(\theta),H^1(I_\theta(G),\F_\ell))\longrightarrow H^4(G(\theta),\F_\ell)
\end{equation}
is injective.
\end{thm}

\begin{rem}
\label{rem:posit}\rm
As $\bos$ is a {\sl Frattini pro-$\ell$ cover} (i.e., $I_\theta(G)$ is contained in the Frattini subgroup of $G$, cf. \S~\ref{ssec:unif}), inflation yields an isomorphism $j^1\colon H^1(G(\theta),\F_\ell)\to H^1(G,\F_\ell)$.
Since $H^\bullet(G,\F_\ell)$ is quadratic, inflation may also be considered 
as a surjective homomorphism of graded $\F_\ell$-algebras
\begin{equation}
\label{eq:qinf}
j^\bullet\colon H^\bullet(G(\theta),\F_\ell)\longrightarrow H^\bullet(G,\F_\ell),
\end{equation}
where the left-side term of \eqref{eq:qinf} is the exterior algebra generated by $H^1(G(\theta),\F_\ell)$ (cf. \S~\ref{ssec:quad NRT}).
By \cite[Theorem 1.4]{pos:k},
$(G,\theta)$ has the Bogomolov-Positselski property provided $H^\bullet(G,\F_\ell)$ is a Koszul $\F_\ell$-algebra
and $\kernel(j^\bullet)$
is a Koszul $H^\bullet(G,\F_\ell)$-module (cf. \cite[\S 3.3]{pos:k}).
Hence the natural question arising in this context is, whether one can express 
$\kernel(d_2^{2,1})$ in terms
of $\Ext^{s,t}_{H^{\bullet}(G,\F_\ell)}(\F_\ell,\kernel(j^\bullet))$, $s\not=t$.
\end{rem}

\begin{acknowledgment} \rm
The authors would like to thank I.~Efrat, A.~Jaikin-Zapirain and P.~Zalesskii as well as K.~Ersoy, I.~Snopce and M.~Vannacci
for interesting and helpful discussions. The authors would also like to thank the anonymous referee for several helpful comments.
\end{acknowledgment}


\section{Oriented pro-$\ell$ groups}
\label{sec:or}


For a pro-$\ell$ group $G$ and a positive integer $n$, $G^n$ will denote the closed subgroup of $G$ generated by the $n$-th powers of all elements of $G$.
Moreover, for two elements $g,h\in G$, we set $${}^gh=ghg^{-1},\qquad\text{and}\qquad[g,h]={}^{g}h\cdot  h^{-1},$$
and for two subgroups $H_1,H_2$ of $G$, $[H_1,H_2]$ will denote the closed subgroup of $G$ generated by all commutators $[g,h]$
with $g\in H_1$ and $h\in H_2$.
In particular, $G'$ will denote the closure of the commutator subgroup of $G$. 


\subsection{$\ell$-Orientations of profinite groups}
\label{ss:or}
Let $\Z_{\ell}$ denote the ring of $\ell$-adic integers, and let $\Z_{\ell}^\times$ denote its group of units.
Note that $\Z_\ell^\times$ is a virtual pro-$\ell$ group, in more detail:
\begin{itemize}
 \item[(a)] if $\ell\neq2$ then the Sylow pro-$\ell$ subgroup of $\Z_\ell^\times$ is $1+\ell\Z_\ell=\{1+\ell\lambda\mid\lambda\in\Z_\ell\}$, which is free pro-$\ell$ cyclic;  
 \item[(b)] if $\ell=2$ then $\Z_2^\times=1+2\Z_2\simeq \Z/2\times(1+4\Z_2)$, and the factor $1+4\Z_2$ is isomorphic to $\Z_2$.
\end{itemize}

An oriented pro-$\ell$ group $(G,\theta)$ is a pro-$\ell$ group $G$ together with a continuous group homomorphism
$\theta\colon G\to\Z_\ell^\times$.
Moreover, $(G,\theta)$ is said to be  {\sl torsion-free}  if $\ell\neq2$, or if $\ell=2$ and $\image(\theta)\subseteq 1+4\Z_2$ --- observe that in a torsion-free oriented pro-$\ell$ group $(G,\theta)$, $G$ need not be a torsion free pro-$\ell$ group, e.g., $(\Z/\ell,\mathbf{1})$ is a torsion-free oriented pro-$\ell$ as $\image(\mathbf{1})=\{1\}$.

Oriented pro-$\ell$ groups where introduced by I.~Efrat in \cite{efrat:small} under the name ``cyclotomic pro-$\ell$ pairs".
For an oriented pro-$\ell$ group $(G,\theta)$, $\Z_\ell(1)$ will denote the continuous left $\ZLG$-module which is isomorphic to $\Z_\ell$ as an abelian pro-$\ell$ group, such that $g\cdot v=\theta(g)\cdot v$ for every $g\in G$ and $v\in\Z_\ell(1)$ (cf. \cite[\S~1]{qw:cyclotomic}).
Conversely, if a pro-$\ell$ group $G$ comes endowed with a continuous left $\ZLG$-module $M$ which is isomorphic to $\Z_\ell$ as an abelian pro-$\ell$ group, then $M$ induces an orientation $\theta\colon G\to\Z_\ell^\times$ by $\theta(g)\cdot v=g\cdot v$ for every $g\in G$ and $v\in M$, such that $M\simeq\Z_\ell(1)$.

The fundamental examples of oriented pro-$\ell$ groups arise in Galois theory (cf. \cite[\S~4]{eq:kummer}).

\begin{exam}\label{ex:Galois}\rm
For a field $\K$, let $\baK$ denote a separable closure of $\K$, and let $\mu_{\ell^\infty}$ denote the 
 group of roots of 1 of $\ell$-power order lying in $\baK$.
 If $\K$ contains a primitive $\ell^{\mathrm{th}}$-root of unity, then $\mu_{\ell^\infty}$ is contained in the maximal pro-$\ell$ extension
$\K(\ell)$ of $\K$.
{As $\mu_{\ell^\infty}\simeq\Z[\frac{1}{\ell}]/\Z$ and $\Aut(\Z[\frac{1}{\ell}]/\Z)$ is isomorphic to $\Z_\ell^\times$,
the action of the {\sl maximal pro-$\ell$ Galois group} $G_{\K}(\ell)=\mathrm{Gal}(\K(\ell)/\K)$ of $\K$
on $\mu_{\ell^\infty}$ fixes the primitive $\ell^{\mathrm{th}}$-roots of unity, and induces the {\sl $\ell$-cyclotomic character} 
\[
 \ttheta_{\K,\ell}\colon G_{\K}(\ell)\longrightarrow \Z_\ell^\times.
\]}
In particular,
$$\sigma(\zeta)=\zeta^{\ttheta_{\K,\ell}(\sigma)}\qquad \text{for all }\sigma\in G_{\K}(\ell),\zeta\in\mu_{\ell^\infty}.$$
Furthermore, one has $\image(\ttheta_{\K,\ell})=1+\ell^f\Z_\ell$ --- where $f$ is the positive integer satisfying
$|\mu_{\ell^\infty}\cap\K^\times|=\ell^f$ --- in case $\mu_{\ell^\infty}\cap\K^\times$ is 
{non-empty and} finite, and 
$\image(\ttheta_{\K,\ell})=\{1\}$ if $\mu_{\ell^\infty}\subseteq \K^\times$.
The continuous $G_{\K}(\ell)$-module $\Z_\ell(1)$ induced by the cyclotomic character is called the 1st Tate twist of $\Z_\ell$ (cf. \cite[Def.~7.3.6]{nsw:cohn}), and for every $n\geq1$, $\Z_\ell(1)/\ell^n$ is isomorphic to the $G_{\K}(\ell)$-module of the $\ell^n$-th roots of 1.
\end{exam}

Note that oriented pro-$\ell$ groups form a category $\Or_\ell$, i.e., for $(G,\theta), (H,\theta^\prime)\in \ob(\Or_\ell)$ a morphism of oriented pro-$\ell$ groups $\phi\colon (G,\theta)\to(H,\theta^\prime)$ is a continuous group homomorphism
$\phi\colon G\to H$ of pro-$\ell$ groups satisfying $\theta^\prime\circ\phi=\theta$.

For an oriented pro-$\ell$ group $(G,\theta)$ one has the following constructions.
\begin{itemize}
 \item[(a)] Let $N$ be a normal subgroup of $G$ such that $N\subseteq \kernel(\theta)$.
 Then one has an oriented pro-$\ell$ group $$(G,\theta)/N:=(G/N,\bar\theta),$$ where $\bar\theta\colon G/N\to\Z_\ell^\times$ is the orientation induced by $\theta$.
 \item[(b)] Let $A$ be an abelian pro-$\ell$ group.
 Then one has an oriented pro-$\ell$ group $$A\rtimes(G,\theta):=(A\rtimes G,\tilde\theta),$$
 where $gag^{-1}=a^{\theta(g)}$ for all $g\in G$ and $a\in A$, and $\tilde\theta=\theta\circ\pi$, where $\pi\colon A\rtimes G\to G$ is the canonical projection.
\end{itemize}


\subsection{The maximal $\theta$-abelian quotient of an oriented pro-$\ell$ group}
\label{ss:maxTh}
Let $(G,\theta)$ be a torsion-free oriented pro-$\ell$ group. 
Then $G/\kernel(\theta)\simeq \image(\theta)$ is torsion-free, and thus either trivial or isomorphic to $\Z_\ell$.
Therefore, the epimorphism $G\twoheadrightarrow G/\kernel(\theta)$ splits, and since $ghg^{-1}\equiv h^{\theta(g)}\bmod K_\theta(G)$ for every $g\in G$ and $h\in\kernel(\theta)$, one concludes that
\begin{equation}\label{eq:K split}
 (G,\theta)/K_\theta(G)\simeq \frac{\kernel(\theta)}{K_\theta(G)}\rtimes (\image(\theta),\mathrm{Id}_{\image(\theta)}).
\end{equation}

\begin{rem}\label{rem:thetabel}\rm
 By \eqref{eq:K split}, if $(G,\theta)$ is a torsion-free $\theta$-abelian oriented pro-$\ell$ group, then it is isomorphic to the oriented pro-$\ell$ group
 $\kernel(\theta)\rtimes (\image(\theta),\mathrm{Id}_{\image(\theta)})$.  
 Conversely, if $A$ is a free abelian pro-$\ell$ group, and $(\bar G,\theta)$ is an oriented pro-$\ell$ group 
 {satisfying
 $\kernel(\theta)=\triv$,} then the oriented pro-$\ell$ group $(G,\tilde\theta)=A\rtimes(\bar G,\theta)$ is $\tilde\theta$-abelian, since $\kernel(\tilde\theta)=A$ is a free abelian pro-$\ell$ group, and 
 {as $ghg^{-1}=h^{\tilde\theta(g)}$ for every $g\in\bar G$ and $h\in A$ and thus} $K_{\tilde\theta}(G)=\{1\}$. 
\end{rem}

Let $(G,\theta)$ be an oriented pro-$\ell$ group. 
{Put
$\bar{G}=G/I_\theta(G)$ and let $\bar\theta\colon\bar{G}\to\Z_\ell^\times$ denote the induced orientation.
Since the quotient $\kernel(\theta)/I_\theta(G)$ is torsion-free (cf. \S~\ref{sec:intro}), the oriented pro-$\ell$ group $(G(\theta),\bar\theta)$ is $\bar\theta$-abelian}.
This group together with the canonical projection
\begin{equation}
\label{eq:Thab}
\pi^{\ab}_{G,\theta}\colon G\longrightarrow G(\theta)
\end{equation}
has the following universal property.

\begin{prop}
\label{prop:maxTh}
Let $(G,\theta)$ be an oriented pro-$\ell$ group,
let $(A,\theta^\circ)$ be an oriented $\theta^\circ$-abelian pro-$\ell$
group, and let $\psi\colon (G,\theta)\to (A,\theta^\circ)$ 
be a continuous homomorphism of
oriented pro-$\ell$ groups. 
Then $\psi$ factors through $\pi^{ab}_{G,\theta}$, i.e., there exists a \textup{(}unique\textup{)} continuous
group homomorphism $$\psi^{\ab}_{G,\theta}\colon (G(\theta),\batheta)\longrightarrow (A,\theta^\circ)$$
satisfying $\psi=\psi^{\ab}_{G,\theta}\circ\pi^{\ab}_{G,\theta}$.
\end{prop}

\begin{proof}
As $\psi$ is a homomorphism of oriented pro-$\ell$ groups, and as 
$(A,\theta^\circ)$ is $\theta^\circ$-abelian, one has
\begin{equation}
\label{eq:max1}
\psi(\kernel(\theta))\subseteq\kernel(\theta^\circ) \qquad\text{and}\qquad
\psi(K_\theta(G))\subseteq K_{\theta^\circ}(A)=\triv.
\end{equation}
As $\kernel(\theta^\circ)$ is torsion-free, this implies that $\psi(I_\theta(G))=\triv$.
Hence the induced homomorphism $\psi^{\ab}_{G,\theta}\colon G(\theta)\to A$
of oriented pro-$\ell$ groups has the desired properties.
\end{proof}




\begin{rem}\label{rem:pthetabel subgroups}\rm
Let $(G,\theta)\simeq A\rtimes ((G,\theta)/\kernel(\theta))$ be a torsion-free $\theta$-abelian oriented pro-$\ell$ group.
Then for every subgroup $H$ of $G$ one has $$H\simeq (H\cap A)\rtimes (H/\kernel(\theta\vert_H)),$$ and thus the oriented pro-$\ell$ group $(H,\theta\vert_H)$ is split $\theta\vert_H$-abelian (cf. \cite[Remark~3.12]{qw:cyclotomic}).
\end{rem}


\subsection{Kummerian oriented pro-$\ell$ groups}
\label{ss:kummer}
Let $(G,\theta)$ be an oriented torsion-free pro-$\ell$ group.
Since $\image(\theta)\subseteq1+\ell\Z_\ell$, the action of $G$ on the quotient $\Z_\ell(1)/\ell$ of the continuous $G$-module $\Z_\ell(1)$ is trivial, i.e., $\Z_\ell(1)/\ell\simeq \F_\ell$ as a trivial left $\ZLG$-module.
{In the proof of the subsequent proposition we will make use of the following}
\begin{fact}
\label{fact:dir}
Let $A$ be an abelian pro-$\ell$ group, and let $B$ be a closed subgroup of $A$ which is a direct summand of $A$
satisfying $B\subseteq A^\ell$. Then $B=\{0\}$.
\end{fact}

\begin{proof}
Let $A=B\oplus C$. Then $A^\ell=B^\ell\oplus C^\ell$, and as $B^\ell\subseteq B$, and $B\cap C=\{0\}$ one concludes
that $B\subseteq B^\ell$, i.e., $B=B^\ell=\Phi(B)$. Hence $B=\{0\}$.
\end{proof}

A torsion-free oriented pro-$\ell$ group $(G,\theta)$ is said to be {\sl Kummerian} if the following equivalent properties are satisfied.

\begin{prop}\label{prop:kummer}
Let $(G,\theta)$ be a torsion-free oriented pro-$\ell$ group. 
Then the following are equivalent:
\begin{itemize}
 \item[(i)] the map
$H^1(G,\Z_\ell(1)/\ell^n)\longrightarrow H^1(G,\F_\ell)$
induced by the epimorphism of discrete left $G$-modules $\Z_\ell(1)/\ell^n\twoheadrightarrow\Z_\ell(1)/\ell\simeq\F_\ell$, is surjective for every $n\geq1$ \textup{(}cf. \cite{eq:kummer}\textup{)}.
 \item[(ii)] The quotient $\kernel(\theta)/K_\theta(G)$ is a free abelian pro-$\ell$ group.
 \item[(iii)] The oriented pro-$\ell$ group $(G,\theta)/K_\theta(G)=(G/K_\theta(G),\bar\theta)$ is $\bar\theta$-abelian.
\item[(iv)] $K_\theta(G)$ is isolated in $\kernel(\theta)$, i.e., $I_\theta(G)=K_\theta(G)$.
\item[(v)] The group $H_{\mathrm{cts}}^2(G,\Z_\ell(1))$ is a torsion-free $\Z_\ell$-module.
\item[(vi)] $I_{\theta}(G)\subseteq\Phi(G)$.
\end{itemize}
\textup{(}Here $H_{\mathrm{cts}}^\ast$ denotes continuous cochain cohomology as defined by J.~Tate in \cite{tate:miln}\textup{)}.
\end{prop}

\begin{proof}
For $G$ finitely generated the equivalences between (i) and (ii) was shown in \cite[Thm.~5.6]{eq:kummer}, and
the equivalence between (ii) and (iii) follows from Remark~\ref{rem:thetabel}. 
For general $G$ the equivalences were shown in \cite[Thm.~1.2]{cq:1smooth}.
The equivalence between (i) and (v) is shown in \cite[Prop.~2.1]{qw:cyclotomic},
and (iii) $\Leftrightarrow$ (iv) is a direct consequence of \eqref{eq:K split} and Remark~\ref{rem:thetabel}.
Hence (i)--(v) are equivalent. As $K_\theta(G)\subseteq\Phi(G)$ one has (iv) $\Rightarrow$ (vi).
Thus it remains to show that (vi) $\Rightarrow$ (iv).
Let $\pi\colon G\to G/\Phi(G)$ denote the canonical projection, and let
\begin{equation}
\pi_\ast\colon \kernel(\theta)/K_\theta(G)\cdot\ker(\theta)^\ell\longrightarrow G/\Phi(G)
\end{equation}
denote the induced map --- note that $K_\theta(G)\ker(\theta)^\ell=\ker(\theta)^\ell[G,\ker(\theta)]$, by \eqref{eq:defK2}. 
As $\image(\theta)$ --- which is isomorphic to either $\Z_\ell$ or $\{1\}$ --- is projective, the 5-term exact sequence associated to the
Hochschild-Serre spectral sequence yields an exact sequence
\begin{equation}
\label{eq:inj2}
\xymatrix{H^1(G,\F_\ell)\ar[r]^-{\pi_\ast^\vee}&H^1(\kernel(\theta),\F_\ell)^G\ar[r]&\{0\}}
\end{equation}
Thus, by Pontrjagin duality, $\pi_\ast$ is  injective.
Note that 
\begin{equation}
\tor(\kernel(\theta)/K_\theta(G))=I_\theta(G)/K_\theta(G)
\end{equation}
is a direct summand of the abelian pro-$\ell$ group $\kernel(\theta)/K_\theta(G)$ (cf. \S~\ref{sec:intro}).
Since $\pi(I_\theta(G))=\{1\}$ by (vi), and since $\pi_\ast$ is injective, one concludes that 
{$I_\theta(G)\subseteq K_\theta(G)\cdot \kernel(\theta)^\ell$. Hence $I_\theta(G)/K_\theta(G)=\{1\}$
by Fact~\ref{fact:dir}.}
\end{proof}

\begin{exam}\label{exam:kummer groups}\rm
\begin{itemize}
 \item[(a)] If $(G,\theta)$ is a torsion-free $\theta$-abelian pro-$\ell$ group, then, by Proposition~\ref{prop:kummer}--(ii), $(G,\theta)$ is Kummerian, as $K_\theta(G)=\{1\}$ and $\kernel(\theta)$ is free abelian by definition.
 \item[(b)] If $G$ is a free pro-$\ell$ group, then by Proposition~\ref{prop:kummer}--(v) the oriented pro-$\ell$ group $(G,\theta)$ is Kummerian for any orientation $\theta\colon G\to\Z_\ell^\times$, as $\mathrm{cd}(G)=1$ (cf. \cite[Prop.~3.5.17]{nsw:cohn}).
\item[(c)] If $(G,\theta)$ is an oriented pro-$\ell$ group with trivial orientation $\theta\equiv\mathbf{1}$,
then $(G,\theta)$ is Kummerian if, and only if, the abelianization $G^{\ab}$ is a free abelian pro-$\ell$ group
(cf. \cite[Example~3.5--(1)]{eq:kummer}).
 \end{itemize}
\end{exam}

The following result is a consequence of Kummer theory (cf. \cite[Thm.~4.2]{eq:kummer}).

\begin{thm}\label{thm:galois}
 Let $\K$ be a field containing a primitive $\ell^{\mathrm{th}}$-root of 1 {\rm (}and also $\sqrt{-1}$ if $\ell=2${\rm )}.
Then $(G_{\K}(\ell),\theta_{\K,\ell})$ is a torsion-free Kummerian oriented pro-$\ell$ group.
\end{thm}

From Proposition~\ref{prop:maxTh} and Proposition~\ref{prop:kummer}--(iv), one concludes the following fact.

\begin{cor}
\label{cor:kummer}
Let $(G,\theta)$ be a Kummerian torsion-free oriented pro-$\ell$ group.
Then $(G/K_\theta(G),\batheta)$ is the maximal $\theta$-abelian quotient of $G$.
\end{cor}


\section{The Bogomolov-Positselski property}\label{sec:bogo}


\subsection{Bogomolov's conjecture}
\label{ssec:bogo}
Let $\K$ be a field containing a primitive $\ell^{\mathrm{th}}$-root of 1 (and also $\sqrt{-1}$ if $\ell=2$), and let $\mathbb{L}=\sqrt[\ell^\infty]{\K}$ denote the compositum of all radical extensions $\K(\sqrt[\ell^n]{a})$, with $a\in\K^\times$ and $n\geq1$, i.e.,
\begin{equation}
\label{eq:FL}
 \mathbb{L}=\sqrt[\ell^\infty]{\K}=\K\left(\:\sqrt[\ell^n]{a}\:\mid\:a\in\K^\times,n\geq1\:\right).
\end{equation}
The maximal pro-$\ell$ Galois group $G_{\mathbb{L}}(\ell)$ of the field $\mathbb{L}$ is equal to the pro-$\ell$ group 
$K_{\ttheta_{\K,\ell}}(G_{\K}(\ell))$ associated to the oriented pro-$\ell$ group 
$(G_{\K}(\ell),\ttheta_{\K,\ell})$ (cf. \cite[Thm.~4.2]{eq:kummer}).
Observe that the $\ell$-cyclotomic character associated to the maximal pro-$\ell$ Galois group of $\mathbb{L}$ is the trivial $\ell$-orientation $\mathbf{1}\colon K_{\ttheta_{\K,\ell}}(G_\K(\ell))\to\{1\}\subseteq\Z_\ell^\times$.

Motivated by a conjecture formulated by F.~Bogomolov in \cite{bogo} --- see Remark~\ref{rem:history Bogoconj} below ---, L.~Positselski stated the following conjecture on the pro-$\ell$ group $G_{\mathbb{L}}(\ell)=K_{\ttheta_{\K,\ell}}(G_\K(\ell))$ (cf. \cite[Conjecture~1.2]{pos:k}).

\begin{conj}\label{conj:bogo}
 Let $\K$ be a field containing a primitive $\ell^{\mathrm{th}}$-root of 1,
and also $\sqrt{-1}$ if $\ell=2$.
Then the maximal pro-$\ell$ Galois group $G_{\mathbb{L}}(\ell)$ of $\mathbb{L}=\sqrt[\ell^\infty]{\K}$ is a free pro-$\ell$ group.
\end{conj}

Conjecture~\ref{conj:bogo} is the motivation for the following definition.

\begin{defin}\label{defi:bogomolov}\rm
A Kummerian oriented pro-$\ell$ group $(G,\theta)$ is said to have the {\sl Bogomolov-Positselski property} if the  subgroup $K_\theta(G)$ is a free pro-$\ell$ group.
\end{defin}

Hence, Conjecture~\ref{conj:bogo} may be restated as follows: if $\K$ is a field containing a primitive $\ell^{\mathrm{th}}$-root of 1 (and also $\sqrt{-1}$ if $\ell=2$), then the oriented pro-$\ell$ group $(G_{\K}(\ell),\ttheta_{\K,\ell})$ has the Bogomolov-Positselski property.

\begin{rem}\label{rem:history Bogoconj}\rm
The original formulation of Bogomolov's conjecture states that if $\K$ is a field containing an algebraically closed field then the (closure of the) commutator subgroup of the Sylow pro-$\ell$ subgroup of the absolute Galois group $G_{\K}$ of $\K$ is a free pro-$\ell$ group. Furthermore, the (closure of the) commutator subgroup of the maximal pro-$\ell$ Galois group $G_{\K}(\ell)$ should be a free pro-$\ell$ group as well (see also \cite[Conjecture~6.2]{bogotschin2} and \cite[\S~3.1.2]{banff}, where the conjecture is stated for function fields).

\noindent
In \cite{pos:k}, Positselski observed that the only essential part of the condition about the algebraically closed subfield of 
$\K$ is that $\K$ should contain all the roots of 1 of $\ell$-power order. Consequently, he formulated the following conjecture (cf. \cite[Conjecture~1.1]{pos:k}): the pro-$\ell$ Sylow subgroup of the absolute Galois group $G_{\mathbb{L}}$, with $\mathbb{L}=\sqrt[\ell^\infty]{\K}$ and $\K$ an arbitrary field, is a free pro-$\ell$ group, i.e., $\cd_\ell(G_{\mathbb{L}})\leq1$ (or, equivalently, $G_{\mathbb{L}}$ is $\ell$-projective, cf. \cite[\S I.3.4, Proposition 16]{serre:gal}).
Note that this conjecture is stronger than Conjecture~\ref{conj:bogo}, and likely hard to approach, while --- as stated by Positselski himself, cf. \cite[\S~1.3]{pos:k} --- the latter is closer to Bogomolov's original conjecture.
\end{rem}

\begin{exam}\label{exam:bogo}\rm
\begin{itemize}
\item[(a)] Let $(G,\theta)$ be a torsion-free $\theta$-abelian oriented pro-$\ell$ group.
Then $(G,\theta)$ is Kummerian (cf. Example~\ref{exam:kummer groups}--(a)), and by Proposition~\ref{prop:kummer}--(iv) one has $$I_\theta(G)=K_\theta(G)=\{1\}.$$ So, $(G,\theta)$ has the Bogomolov-Positselski property. 

\item[(b)] Let $(G,\theta)$ be an oriented pro-$\ell$ group with $G$ being a free pro-$\ell$ group.
Then $(G,\theta)$ is Kummerian by Example~\ref{exam:kummer groups}--(b), and it has the Bogomolov-Positselski property as every closed subgroup of $G$ is again a free pro-$\ell$ group.
\item[(c)] Let 
\[\begin{split}
   G &= \langle\: x,y,z\:\mid\: [x,y]=z, [x,z]=[y,z]=1\:\rangle\\
   &= \left\{\left(\begin{array}{ccc} 1 & a & c \\ 0 & 1 & b \\ 0 & 0 & 1 \end{array} \right)\:\mid\:a,b,c\in\Z_\ell\right\}
  \end{split}\]
  be the {\sl Heisenberg group over $\Z_\ell$}, and set $(G,\bone)$, where $\bone\colon G\to\Z_\ell^\times$ is the trivial orientation.
Then $K_\theta(G)=G^\prime\simeq \Z_\ell$ is the cyclic pro-$\ell$ subgroup generated by $z$, and  $G^{\ab}\simeq\Z_\ell^2$.
Hence $(G,\theta)$ is Kummerian by Example~\ref{exam:kummer groups}--(c), and it has the Bogomolov-Positselski property. 
Nevertheless, $G$ does not occur as the maximal pro-$\ell$ Galois group of any field containing $\mu_{\ell^\infty}$ (cf. \cite[Ex.~5.4]{cq:1smoothBK}).
\end{itemize}
\end{exam}




\subsection{Self-isolated pro-$\ell$ groups and Frattini pro-$\ell$ covers}\label{ssec:unif}
Let $G$ be a pro-$\ell$ group, and let $H\subseteq G$ be a subgroup.
The {\sl isolator of $H$} is the subgroup
\[
 \Iso(H)=\cl\left(\left\langle\:g\in G\:\mid\:g^n\in H\text{ for some }n\geq 1\:\right\rangle\right)
\]
(cf. \cite[\S~66]{kur:grp}).
We say that $H$ is {\sl self-isolated} if $\Iso(H)=H$.
In particular, if $N$ is a normal subgroup of $G$, then $G$ is self-isolated if, and only if, the quotient $G/N$ is a torsion-free pro-$\ell$ group.
The following fact is almost straightforward.

\begin{fact}\label{fact:selfiso}
 Let $(G,\theta)$ be a torsion-free $\theta$-abelian oriented pro-$\ell$ group.
 Let $N$ be a normal subgroup of $G$ contained in both $\kernel(\theta)$ and $\Phi(G)$.
 If $N$ is self-isolated, then $N=\{1\}$.
\end{fact}

\begin{proof}
{By Remark~\ref{rem:thetabel}, $\Phi(G)\cap\kernel(\theta)=\kernel(\theta)^\ell$. 
 As $N\subseteq\kernel(\theta)$ is an isolated subgroup, it is a direct summand of $\kernel(\theta)$.
 Thus by Fact~\ref{fact:dir},  $N$ is trivial.}
\end{proof}

Fact~\ref{fact:selfiso} has the following consequence.

\begin{prop}
\label{prop:fratK}
Let $(G,\theta)$ be a torsion-free Kummerian oriented pro-$\ell$ group. Let $N\trianglelefteq G$ be a closed normal, self-isolated, subgroup of $G$ contained in $\kernel(\theta)$ satisfying 
$$K_\theta(G)\subseteq N\subseteq\Phi(G).$$ 
Then $N=K_\theta(G)$.
\end{prop}

A {\sl Frattini pro-$\ell$ cover} of pro-$\ell$ groups is a short exact sequence of pro-$\ell$ groups 
\begin{equation}\label{eq:ses Frattinicover}
 \xymatrix{\{1\} \ar[r] & N \ar[r]  & G \ar[r]^-{\tau}&\bar G \ar[r]& \{1\}}
\end{equation}
satisfying $N\subseteq \Phi(G)$. One also says that $\tau\colon G\to\baG$
is a Frattini pro-$\ell$ cover of $\bar G$.
One may characterize those pro-$\ell$ groups which may be completed into Kummerian oriented pro-$\ell$ groups with the Bogomolov-Positselski property as follows.

\begin{thm}
\label{prop:Kummer fratcov}
A pro-$\ell$ group $G$ may be completed into a Kummerian oriented pro-$\ell$ group $(G,\theta)$
with the Bogomolov-Positselski property if, and only if, $G$ is a Frattini pro-$\ell$ cover \eqref{eq:ses Frattinicover}
of {$\bar G$, where $(\bar G,\bar\theta)$ is a $\bar\theta$-abelian oriented pro-$\ell$ group and $N$ is a free pro-$\ell$ group.}
\end{thm}

\begin{proof}
If $(G,\theta)$ is Kummerian with the Bogomolov-Positselski property, then, 
{by Proposition~\ref{prop:kummer}}, $(G/K_\theta(G),\bar\theta)=(G,\theta)/K_{\theta}(G)$ is $\bar\theta$-abelian and $N=K_{\theta}(G)$ is a free pro-$\ell$ group by Definition~\ref{defi:bogomolov}. 
This shows one implication.

Conversely, if $(\bar G,\bar\theta)$ is $\bar\theta$-abelian, then the epimorphism of oriented pro-$\ell$ groups $(G,\theta)\to(\bar G,\bar\theta)$ factors through $(G,\theta)/I_{\theta}(G)$ by Proposition~\ref{prop:maxTh}. 
Hence $I_\theta(G)\subseteq N$, while $N\subseteq\Phi(G)$ by hypothesis, thus $(G,\theta)$ is Kummerian by Proposition~\ref{prop:kummer}:(vi).
Thus, $I_\theta(G)=K_\theta(G)$ by Proposition~\ref{prop:kummer}:(iv), and since 
$$K_\theta(G)=I_\theta(G)\subseteq N\subseteq\Phi(G),$$ Proposition~\ref{prop:fratK} yields $N=K_\theta(G)$, i.e., $(G,\theta)$ has the Bogomolov-Positselski property.
\end{proof}


\section{The Bogomolov-Positselski property and cohomology}


\subsection{Quadratic cohomology and the Norm Residue Theorem}
\label{ssec:quad NRT}

Let $G$ be a pro-$\ell$ group.
The cohomology groups $H^n(G,\F_\ell)$, $n\geq1$, where $\F_\ell$ is the trivial $G$-module 
isomorphic --- as abelian group --- to $\F_\ell=\Z/\ell\Z$, come endowed with the bilinear {\sl cup-product}
\[
 H^s(G,\F_\ell)\times H^t(G,\F_\ell)\overset{\cup}{\longrightarrow} H^{s+t}(G,\F_\ell),\qquad s,t\geq0,
\]
which is associative and graded-commutative, i.e., $\beta\cup\alpha=(-1)^{st}\alpha\cup\beta$ for $\alpha\in H^s(G,\F_\ell)$ and $\beta\in H^t(G,\F_\ell)$ (cf. \cite[Ch.~I, \S~4]{nsw:cohn}).
Thus, 
\[
 H^\bullet(G,\F_\ell)=\coprod_{n\geq0}H^n(G,\F_\ell)
\]
is a connected $\dbN_0$-graded, graded-commutative, associative $\F_\ell$-algebra.

For an $\F_\ell$-vector space $V$, let $\boT^\bullet V$ denote
the $\F_\ell$-tensor algebra, i.e., 
\begin{equation}\label{eq:tenalg}
\boT^\bullet V=\coprod_{n\in\dbN_0} \boT^n V\,\qquad \text{where}\qquad  \boT^n V=V^{\otimes^n}.
\end{equation}
The $\dbN_0$-graded associative $\F_\ell$-algebra $\boA_\bullet$ is said to be generated in degree $1$, if the
canonical homomorphism $\phi_\bullet\colon \boT^\bullet \boA_1\to \boA_\bullet$ of $\dbN_0$-graded associative $\F_\ell$-algebras is surjective. Moreover, $\boA_\bullet$ is said to be quadratic, if it is $1$-generated
and $\kernel(\phi_\bullet)=\langle \kernel(\phi_2)\rangle$, i.e., the ideal $\kernel(\phi_\bullet)$ is generated
in degree $2$.

\begin{defin}\label{defin:Hquad}\rm
A pro-$\ell$ group $G$ is said to be {\sl $H^\bullet$-quadratic} if $H^\bullet(G,\F_\ell)$ is a quadratic algebra. \end{defin}

For an $\F_\ell$-vector space $V$, let 
$\bfLam^\bullet V=\boT^\bullet V/\langle\, v\otimes v\mid v\in V\,\rangle$
denote the {\sl exterior $\F_\ell$-algebra} spanned by $V$, and 
$\boS^\bullet V=\boT^\bullet V/\langle\, v\otimes w-w\otimes v\mid v,w\in V\,\rangle$ denote the
{\sl symmetric $\F_\ell$-algebra} spanned by $V$.
Then $G$ is $H^\bullet$-quadratic if the cup-product induces an isomorphism of graded $\F_\ell$-algebras
\begin{equation}\label{eq:tensor iso}
 \bfXi^\bullet H^1(G,\F_\ell)/\langle\, W\,\rangle\overset{\sim}{\longrightarrow}H^\bullet(G,\F_\ell),
\end{equation}
where $\bfXi^\bullet=\bfLam^\bullet$ if $\ell$ is odd, and $\bfXi^\bullet=\boS^\bullet$ if $\ell=2$. Moreover,
\begin{equation}
\label{eq:defW}
W=\kernel\left(\Xi^2(H^1(G,\F_\ell))\overset{\cup}{\longrightarrow} H^2(G,\F_\ell)\right).
\end{equation}
By the Norm Residue Theorem, if the field $\K$ contains a primitive $\ell^{\mathrm{th}}$-root of unity,
then the maximal pro-$\ell$ Galois group $G_{\K}(\ell)$ is $H^\bullet$-quadratic (cf. \cite{cq:bk} or \cite{qw:cyclotomic}).

\begin{rem}\label{rem:bockstein}\rm
Let $\ell=2$ and let $G$ be a pro-$2$ group.
Then one has $\alpha\cup\alpha=0$ for every $\alpha\in H^1(G,\F_2)$ if, and only if, the map 
$$H^1(G,\Z/4)\longrightarrow H^1(G,\F_2),$$ induced by the epimorphism of trivial $G$-modules $\Z/4\to\F_2$, is surjective
(cf. \cite[Fact~7.1]{qw:cyclotomic}).
 In particular, if $(G,\theta)$ is a torsion-free Kummerian oriented pro-$2$ group, one concludes that
 $\alpha\cup\alpha=0$ for all $\alpha\in H^1(G,\F_2)$.
 This is the case for $(G_{\K}(2),\ttheta_{\K,2})$, with $\K$ a field containing $\sqrt{-1}$, i.e.,  $H^\bullet(G_{\K}(2),\F_2)$ is quadratic and also a quotient of the exterior algebra $\bfLam^\bullet H^1(G_{\K}(2),\F_2)$.
 \end{rem}

\begin{exam}\label{ex:cohom locunif}\rm
 Let $(G,\theta)$ be torsion-free $\theta$-abelian oriented pro-$\ell$ group.
 Then $G$ is a torsion free powerful pro-$\ell$ group (cf. \cite[Ch.~4, \S~1]{ddsms}), and 
 $$G\simeq\varprojlim_{i\in I} A_i\rtimes \image(\theta)$$ for finitely generated free abelian pro-$\ell$ groups $A_i$.
 Thus by M.~Lazard's theorem (cf. \cite{lazard}) one has $\bfLam_\bullet H^1(G,\F_\ell)$ (see, e.g., \cite[Thm.~3.13]{qw:cyclotomic}), and hence $G$ is $H^\bullet$-quadratic. 
 \end{exam}

\subsection{Quadratic cohomology and {the Bogomolov-Positselski property}}
\label{ss:quadkum}
Let $(G,\theta)$ be a torsion-free Kummerian oriented pro-$\ell$ group.
The short exact sequence of pro-$\ell$ groups 
\begin{equation}
\label{eq:ses GK}
 \xymatrix{ \{1\} \ar[r] & I_\theta(G) \ar[r] &G \ar[r] & G(\theta) \ar[r] &\{1\}}
\end{equation}
induces the 5-terms exact sequence in cohomology 
\begin{equation}\label{eq:5tes}
 \begin{tikzpicture}[descr/.style={fill=white,inner sep=2pt}]
        \matrix (m) [
            matrix of math nodes,
            row sep=3.5em,
            column sep=3.8em,
            text height=1.5ex, text depth=0.25ex
        ]
        {  0 & H^1( G(\theta),\F_\ell) & H^1(G,\F_\ell) & H^1(I_\theta(G),\F_\ell)^{ G(\theta)} \\
            & H^2( G(\theta),\F_\ell) & H^2(G,\F_\ell) & \\
           };

        \path[overlay,->, font=\scriptsize,>=latex]
        (m-1-1) edge  (m-1-2) 
        (m-1-2) edge node[auto] {$\Inf_{ G(\theta),G}^1$} (m-1-3) 
        (m-1-3) edge node[auto] {$\Res^1_{G,I_\theta(G)}$} (m-1-4)
        (m-1-4) edge[out=355,in=175] node[descr,yshift=0.3ex] {$ d_2^{0,1}$} (m-2-2)
        (m-2-2) edge node[auto] {$\Inf_{ G(\theta),G}^2$} (m-2-3);
\end{tikzpicture}\end{equation}
(cf. \cite[Prop.~1.6.7]{nsw:cohn}). As $(G,\theta)$ is Kummerian, one has
 {$I_\theta(G)$}$=K_\theta(G)\subseteq\Phi(G)$ (cf. Proposition~\ref{prop:kummer}(iv)). 
 {Hence $\Inf_{ G(\theta),G}^1$ is an isomorphism and} $\Res^1_{G,I_\theta(G)}$ is the 0-map.
As $(G(\theta),\bar\theta)=(G,\theta)/I_\theta(G)$ is $\bar\theta$-abelian, 
 one has 
\begin{equation}\label{eq:Lambda Gtheta}
 H^\bullet( G(\theta),\F_\ell)\simeq\bfLam^\bullet H^1(G,\F_\ell)
\end{equation}
(cf. Example~\ref{ex:cohom locunif}).
If in addition $G$ is $H^\bullet$-quadratic, then $H^\bullet(G,\F_\ell)$ is a quotient of $\bfLam^\bullet H^1(G,\F_\ell)$ (cf. Remark~\ref{rem:bockstein}).
In particular, the inflation map $\psi^\bullet=\Inf_{ G(\theta),G}^\bullet$ induces a surjective homomorphism of $\N_0$-graded $\F_\ell$-algebras
\begin{equation}\label{eq:psi n}
 \xymatrix{ H^\bullet( G(\theta),\F_\ell)\simeq\Lambda^\bullet H^1(G,\F_\ell)\ar@{->>}[r]^-{\psi_\bullet} &H^\bullet(G,\F_\ell)}
\end{equation}
satisfying
\begin{equation}
 \kernel(\psi_n)\simeq \kernel(\psi_2)\wedge\left(\Lambda^{n-2}H^1(G,\F_\ell)\right)\quad\text{for all }n\geq2.
\end{equation}
Since $\Res^1_{G,K_\theta(G)}$ is trivial,  one concludes from \eqref{eq:5tes} that $d_2^{0,1}$ is injective, 
$\image(d_2^{0,1})=\kernel(\psi_2)$,  and $H^2(G,\F_\ell)\simeq H^2( G(\theta),\F_\ell)/\image(d_2^{0,1})$.
{Thus, as $H^\bullet(G,\F_\ell)$ is quadratic, one has
\begin{equation}
\label{eq:quadG}
H^\bullet(G,\F_\ell)\simeq H^\bullet( G(\theta),\F_\ell)/\langle\, \image(d_2^{0,1}) \,\rangle.
\end{equation}}


\subsection{A cohomological criterion}
\label{ss:HSspec}
Let $(G,\theta)$ be a Kummerian torsion-free oriented pro-$\ell$ group {which is} $H^\bullet$-quadratic.
Let $(E_r^{s,t},d_r^{s,t})$ denote the Hochschild-Serre spectral sequence with coefficients in $\F_\ell$ associated to the short exact sequence \eqref{eq:ses GK}, {i.e.,}
\begin{equation}
\label{eq:boeq2}
  E_2^{s,t}=H^s( G(\theta),H^t({I}_\theta(G),\F_\ell))\Longrightarrow  E_\infty^{s,t},\qquad s,t\geq0,
\end{equation}
with differentials $d_r^{s,t}\colon E_r^{s,t}\to E_r^{s+r,t-r+1}$ satisfying $d_r\circ d_r=0$
(cf. \cite[Ch.~II, \S~4]{nsw:cohn}).
In particular, by \eqref{eq:psi n} one has $E_2^{\bullet,0}\simeq\Lambda^\bullet H^1(G,\F_\ell)$.
For $\alpha\in E_2^{s,0}=H^s( G(\theta),\F_\ell)$, $s\geq0$, and 
$\beta\in E_2^{0,1}=H^1({I}_\theta(G),\F_\ell)^{ G(\theta)}$, one has 
\begin{equation}\label{eq:cup product E}
\begin{split}
&\alpha\cup\beta\in H^s( G(\theta),\F_\ell\otimes H^1({I}_\theta(G),\F_\ell))=E_2^{s,1},\\
&  d_2^{s,1}(\alpha\cup\beta)=(-1)^{s+1}\alpha\cup d_2^{0,1}(\beta)\in E_2^{s+2,0}
\end{split}
\end{equation}
(cf. \cite[Ch.~II, Ex.~4.5]{nsw:cohn}).

\begin{prop}\label{lem:E}
 Let $(G,\theta)$ be a  torsion-free Kummerian oriented pro-$\ell$ group with $G$ being $H^\bullet$-quadratic.
 Then
 \begin{itemize}
  \item[(i)] $E_\infty^{s,t}$ is concentrated on the 0th line, i.e., $E_\infty^{s,t}=0$ for every $s\geq0$ and $t\geq1$;
  \item[(ii)] $E_3^{s,0}\simeq E_\infty^{s,0}\simeq H^s(G,\F_\ell)$ for every $s\geq0$.
 \end{itemize}
\end{prop}

\begin{proof}
Since $(G,\theta)$ is Kummerian, by \eqref{eq:Lambda Gtheta} one has $E_2^{\bullet,0}\simeq \Lambda^\bullet H^1(G,\F_\ell)$.
For every $t\geq 0$ there exists a descending separating filtration $(F^kH^t(G,\F_\ell))_{0\leq k\leq t}$ satisfying
$F^0H^t(G,\F_\ell)=H^t(G,\F_\ell)$ and
\begin{equation}
\label{eq:filinfty}
F^sH^{s+t}(G,\F_\ell)/F^{s+1}H^{s+t}(G,\F_\ell)\simeq E^{s,t}_\infty
\end{equation}
where $F^{s+t+1}H^{s+t}(G,\F_\ell)=\{0\}$ (cf. \cite[p.~99]{bens:2coh}).
By \cite[Ch.~II, \S~4, Ex.~1]{nsw:cohn},  the composite of the maps
\[
 \xymatrix{ E_2^{s,0}=H^s(G(\theta),\F_p)\ar@{->>}[r] & E_3^{s,0}\ar@{->>}[r] & \cdots \ar@{->>}[r] & E_\infty^{s,0}\ar@{>->}[r]& H^s(G,\F_p)}
\]
{is} the  {\sl $s$-th left edge morphism} (cf. \cite[p.~99]{nsw:cohn}) 
{and hence coincides with  the inflation map}
$\Inf_{G(\theta),G}^s$, which is surjective by \eqref{eq:psi n}. Thus $F^0H^t(G,\F_\ell)=H^t(G,\F_\ell)$ for all $t\geq 0$, i.e.,
 $E_\infty^{\bullet,0}\simeq H^\bullet(G,\F_p)$, and consequently $E_\infty^{k,t}=0$ for every $1\leq k\leq t$.
This shows (i).

By \eqref{eq:cup product E}, one has canonical homomorphisms of $\N_0$-graded $\F_\ell$-algebras
\begin{equation}
\label{eq:morst}
\begin{gathered}
\sigma^\bullet\colon H^\bullet( G(\theta),\F_\ell)/\langle\, \image(d_2^{0,1}) \,\rangle \longrightarrow E_3^{\bullet,0},\\
\tau^\bullet\colon E_3^{\bullet,0}\longrightarrow H^\bullet( G,\F_\ell).
\end{gathered}
\end{equation}
{Moreover, $\sigma^\bullet$ and $\tau^\bullet$ are surjective, $\sigma^k$ and $\tau^k$ are isomorphisms for 
$k\in\{0,1,2\}$,
and their composition is an isomorphism of quadratic $\F_\ell$-algebras by \eqref{eq:quadG}. Thus
$\sigma^\bullet$ and $\tau^\bullet$ are isomorphisms which shows (ii).}
\end{proof}

Let $(G,\theta)$ be a Kummerian torsion-free oriented pro-$\ell$ group, and 
{pu}t $K_\theta(G)^{\mathrm{ab}}=K_\theta(G)/K_\theta(G)'$.
Recall that if $\K$ is a field containing a primitive $\ell^{\mathrm{th}}$-root of 1 and $(G,\theta)=(G_{\K}(\ell),\tilde\theta_{\K,\ell})$, then $K_\theta(G)^{\mathrm{ab}}$ is a free abelian pro-$\ell$ group, as the oriented pro-$\ell$ group $(K_\theta(G),\theta\vert_{K_\theta(G)})$ is again Kummerian, and since $\theta\vert_{K_\theta(G)}$ is trivial.
The short exact sequence of pro-$\ell$ groups
\begin{equation}\label{eq:ext}
  \xymatrix{ \{1\} \ar[r]& K_\theta(G)^{\mathrm{ab}} \ar[r]^-{\iota} & G/K_\theta(G)'\ar[r]^-{\pi}&  G(\theta)\ar[r]& \{1\} },
\end{equation}
where $ G(\theta)=G/K_\theta(G)$, defines a cohomology class $u\in H_{\mathrm{cts}}^2( G(\theta),K_\theta(G)^{\mathrm{ab}})$ (cf. \cite[p.~143]{nsw:cohn}), where $K_\theta(G)^{\mathrm{ab}}$ is considered as a topological left 
$\Z_\ell\dbl G(\theta)\dbr$-module and $H_{\mathrm{cts}}^\ast$ denotes continuous cochain cohomology (cf. \cite[Ch.~II, \S~7]{nsw:cohn}).
Since $[G,K_\theta(G)]\subseteq\Phi(G)$, one has 
\[\mathrm{Hom}(K_\theta(G),\F_\ell)=\mathrm{Hom}(K_\theta(G),\F_\ell)^{ G(\theta)}=E_2^{0,1}.
\]
Thus, the pairing 
\[\begin{split}
  & K_\theta(G)^{\mathrm{ab}}\times E_2^{0,1}\longrightarrow\F_\ell,\\
  & (hK_\theta(G)',\beta)\longmapsto \beta(h),\qquad \text{for }h\in K_\theta(G),
\end{split} 
\]
induces a map
\begin{equation}
\phi_u\colon E_2^{2,1}=H^2( G(\theta),\mathrm{Hom}(K_\theta(G),\F_\ell))\longrightarrow E_2^{4,0}=H^4( G(\theta),\F_\ell)
\end{equation}
given by $\phi_u(\alpha)= u\cup \alpha$ (cf. \cite[p.~114]{nsw:cohn}).


\begin{thm}
\label{prop:bogo2}
 Let $(G,\theta)$ be a Kummerian torsion-free oriented pro-$\ell$ group with $G$ an $H^\bullet$-quadratic pro-$\ell$ group.
 Then the following are equivalent.
\begin{itemize}
\item[(i)] $(G,\theta)$ has the Bogomolov-Positselski property;
\item[(ii)] the differential map $d_2^{2,1}\colon E_2^{2,1}\to E_2^{4,0}\simeq\Lambda^4H^1( G(\theta),\F_\ell)$ is injective;
\item[(iii)] the map $\phi_u$ is injective, i.e., $u\cup\alpha\neq0$ for every non-trivial $\alpha\in E_2^{2,1}$.
\end{itemize}
If these conditions hold, then the spectral sequence $E_2^{s,t}\Rightarrow E_\infty^{s,t}$
collapses at the $E_3$-page, i.e., $E_3=E_\infty$.
\end{thm}

\begin{proof}
By Proposition~\ref{lem:E}(ii), for every $s\geq0$ one has $E_3^{s,0}\simeq E_4^{s,0}\simeq\ldots\simeq E_\infty^{s,0}$.
Since, by definition, $E_4^{s,0}=E_3^{s,0}/\image(d_3^{s-3,2})$, one {conclude}s that the maps
$$ d_3^{s-3,2}\colon E_3^{s-3,2}\longrightarrow E_3^{s,0}\simeq H^s(G,\F_\ell)$$ {must be} the 0-maps for every $s\geq3$.
In particular, $E_4^{0,2}=\kernel(d_3^{0,2})$ is equal to $E_3^{0,2}$, which is $\kernel(d_2^{0,2})$ by definition. 
As $E_r^{s,t}$ is a first-quadrant spectral sequence, one has  $E_{r+1}^{0,2}=\ker(d_r^{0,2})$
and the map $d_r^{0,2}\colon E_r^{0,2}\to E_r^{r,3-r}=0$ is the 0-map for every $r\geq4$.
This implies that $E_3^{0,2}=E_4^{0,2}=\ldots=E_\infty^{0,2}$.
Thus, applying Proposition~\ref{lem:E}(i), yields
\begin{equation}\label{eq:thm E 1}
 0=E_\infty^{0,2}=E_3^{0,2}=\kernel(d_2^{0,2}),
\end{equation}
i.e., $d_2^{0,2}\colon E_2^{0,2}\to E_2^{2,1}$ is injective.

Moreover, one has $E_3^{2,1}=E_4^{2,1}=E_\infty^{2,1}$, as $E_{r+1}^{2,1}=\kernel(d_r^{2,1})/\image(d_r^{2-r,r})$ and both maps 
$$d_r^{2,1}\colon E_r^{2,1}\longrightarrow E_r^{2+r,2-r}=0\qquad\text{and}\qquad 
d_r^{2-r,r}\colon E_r^{2-r,r}=0\longrightarrow E_r^{2,1}$$
are the 0-maps for every $r\geq3$.
Applying Proposition~\ref{lem:E}(i) {again} yields
\begin{equation}\label{eq:thm E 2}
 0=E_\infty^{2,1}=E_3^{2,1}=\kernel(d_2^{2,1})/\image(d_2^{0,2}),
\end{equation}
i.e., $\kernel(d_2^{2,1})=\image(d_2^{0,2})$.

Thus, if $(G,\theta)$ has the Bogomolov-Positselski property, then {$I_\theta(G)=K_\theta(G)$} is a free pro-$\ell$ group.
Then $H^t(K_\theta(G),\F_\ell)=0$ for every $t\geq2$ (cf. \cite[Prop.~3.5.17]{nsw:cohn}), and thus $E_r^{0,t}=0$
for all $r\geq 2$ and $t\geq2$.
In particular, the map $d_2^{0,2}\colon H^2(I_\theta(G),\F_\ell)^{ G(\theta)}\to E_2^{2,1}$ is trivial, and {hence} by \eqref{eq:thm E 2}, one has $\kernel(d_2^{2,1})=0$. 
This proves the implication (i)$\Rightarrow$(ii).

Conversely, {if} $d_2^{2,1}$ is injective, then,
 by \eqref{eq:thm E 2}, one has $\image(d_2^{0,2})=\kernel(d_2^{2,1})=0$. Since $d_2^{0,2}$ is injective by \eqref{eq:thm E 1}, this implies that
 $E_2^{0,2}=H^2(I_\theta(G),\F_\ell)^{ G(\theta)}{=0}$.
Since $ G$ is a pro-$\ell$ group, the equality $H^2(I_\theta(G),\F_\ell)^{ G(\theta)}=0$ implies that $H^2(I_\theta(G),\F_\ell)=0$, and thus $I_\theta(G)$ is free by \cite[Prop.~3.5.17]{nsw:cohn}. This proves the implication (ii)$\Rightarrow$(i).
The equivalence between (ii) and (iii) follows from \cite[Thm.~2.4.4]{nsw:cohn}.

Finally, if $I_\theta(G)$ is a free pro-$\ell$ group, one has $E_r^{s,t}=0$  for all $s\geq0$, $t\geq 2$, and $r\geq2$.
Hence, all maps $d_3^{s,t}$ are trivial, for all $s,t\geq0$, so that $E_3^{s,t}=E_\infty^{s,t}$.
\end{proof}

\begin{question}\label{ques:E3}\rm
 Let $(G,\theta)$ be a Kummerian torsion-free pro-$\ell$ group with $G$ being an $H^\bullet$-quadratic pro-$\ell$ group, and let $(E_r^{s,t},d_r^{s,t})$ be the Hochschild-Serre spectral sequence associated to \eqref{eq:ses GK}.
 By Proposition~\ref{lem:E}, for every $s\geq0$ one has $E_3^{s,0}\simeq E_\infty^{s,0}$, and $E_\infty^{s,t}=0$ for $s\geq0$ and $t\geq1$.
 Moreover, by Theorem~\ref{prop:bogo2}, if $(G,\theta)$ has the Bogomolov-Positselski property, then 
 \begin{equation}\label{eq:isomorph condition}
  E_3^{s,t}\simeq E_\infty^{s,t}\qquad\text{for every }s,t\geq0,
 \end{equation}
 i.e., $E_r^{s,t}$ collapses at the $E_3$-page.
It would be interesting to understand whether \eqref{eq:isomorph condition} implies the Bogomolov-Positselski property for $(G,\theta)$.  We suspect that the answer 
should be affirmative. However, we could not find any evidence for this speculation.
\end{question}

\begin{rem}\rm
Let $\K$ be a field containing a primitive $\ell^{\mathrm{th}}$-root of unity (and also $\sqrt{-1}$ if $\ell=2$), put $\mathbb{L}=\sqrt[\ell^\infty]{\K}$ and consider the torsion-free Kummerian oriented pro-$\ell$ group 
$(G_{\K}(\ell),\ttheta_{\K,\ell})$. 
The oriented pro-$\ell$ group $(I_{\ttheta_{\K,\ell}}(G_{\K}(\ell)),\bone)$ 
is again Kummerian and torsion free, and thus one has 
\begin{align}
I_{\ttheta_{\K,\ell}}(G_{\K}(\ell))^\prime&
=K_{\bone}(G_{\LL}(\ell))=G_{\sqrt[\ell^\infty]{\LL}}(\ell)\label{eq:LK1}\\
K_{\ttheta_{\K,\ell}}(G_{\K})^{\ab}&=G_{\mathbb{L}}(\ell)^{\ab}=\Gal(\sqrt[\ell^\infty]{\LL}/\LL),
\label{eq:LK2}
\end{align}
where the latter is a free abelian pro-$\ell$ group (cf. Example~\ref{exam:kummer groups}--(c)).
Hence, the short exact sequence \eqref{eq:ext} translates into
\begin{equation}\label{eq:ext Gal}
 \xymatrix{ \{1\} \ar[r]& \Gal(\sqrt[\ell^\infty]{\mathbb{L}}/\mathbb{L}) \ar[r]^-{\iota} &\Gal(\sqrt[\ell^\infty]{\mathbb{L}}/\K) \ar[r]^-{\pi}& \Gal(\mathbb{L}/\K)\ar[r]& \{1\} }.
\end{equation}
Recall that by Kummer theory one has an isomorphism of (discrete) $\ell$-elementary abelian groups $H^1(\Gal(\sqrt[\ell^\infty]{\mathbb{L}}/\mathbb{L}),\F_\ell)\simeq\mathbb{L}^\times/(\mathbb{L}^\times)^\ell$, where $\mathbb{L}^\times=\mathbb{L}\smallsetminus \{0\}$ denotes the multiplicative group of the field $\mathbb{L}$. 
Then by Theorem~\ref{prop:bogo2} the cohomology element $u\in H_{\mathrm{cts}}^2(\Gal(\mathbb{L}/\K),\Gal(\sqrt[p^\infty]{\mathbb{L}}/\mathbb{L}))$ associated to the extension of pro-$\ell$ groups \eqref{eq:ext Gal} induces a homomorphism
\[
 \phi_{u,\mathbb{L}}\colon H^2\left(\Gal(\mathbb{L}/\K),\mathbb{L}^\times/(\mathbb{L}^\times)^\ell\right)\longrightarrow
 H^4(\Gal(\mathbb{L}/\K),\F_\ell)
\]
which is injective if, and only if, $\mathbb{L}$ satisfies Conjecture~\ref{conj:bogo}.
In view of Theorem~\ref{prop:bogo2}, the knowledge of the structure of $\mathbb{L}^\times/(\mathbb{L}^\times)^\ell$ as continuous $\Gal(\mathbb{L}/\K)$-module, or an arithmetic interpretation of the map $\phi_{u,\mathbb{L}}$, may contribute to the solution of Conjecture~\ref{conj:bogo}.
\end{rem}


\section{Oriented pro-$\ell$ groups of elementary type}
\label{sec:operations}


\subsection{Demushkin groups and one-relator pro-$\ell$ groups} 
\label{ssec:demushkin}

A {\sl Demushkin group} is a Poincar\'e duality pro-$\ell$ group of dimension 2, namely, a pro-$\ell$ group $G$
whose $\F_\ell$-cohomology satisfies the following conditions:
\begin{itemize}
 \item[(i)] $\dim(H^1(G,\F_\ell))<\infty$;
 \item[(ii)] $H^2(G,\F_\ell)\simeq\F_\ell$;
 \item[(iii)] cup-product induces a perfect pairing $H^1(G,\F_\ell)\times H^1(G,\F_\ell)\to H^2(G,\F_\ell)$
\end{itemize}
(cf. \cite[Def.~3.9.9]{nsw:cohn}).
Note that by condition (ii) such a pro-$\ell$ group $G$ has a single defining relation, namely, $G$ may be defined as the quotient $F/N$ of a free pro-$\ell$ group $F$ over a normal subgroup $N\subseteq F$ generated as a normal subgroup of $F$ by a single element contained in $\Phi(F)$ (cf., e.g., \cite[p.~231--232]{nsw:cohn}).


A Demushkin group comes equipped with a distinguished orientation $\eth_G\colon G\to\Z_\ell^\times$, induced by the 
action of $G$ on its dualizing module, described in \cite[Thm.~4]{labute:demushkin}.
The orientation $\eth_G\colon G\to\Z_\ell^\times$ is the only orientation which completes $G$ into a Kummerian oriented pro-$\ell$ group $(G,\eth_G)$ (cf. \cite[Prop.~5.2]{qw:cyclotomic}).
The oriented pro-$\ell$ group $(G,\eth_G)$ enjoys also the Bogomolov-Positselski property.

\begin{thm}
\label{thm:demushkin}
Let $G$ be a Demushkin group, endowed with the canonical orientation $\eth_G\colon G\to\Z_\ell^\times$, and suppose that $\image(\eth_G)\subseteq1+4\Z_2$ if $\ell=2$.
Then the oriented pro-$\ell$ group $(G,\eth_G)$ has the Bogomolov-Positselski property.
\end{thm}

\begin{proof}
Since $(G,\eth_G)$ is Kummerian, by \cite[Prop.~5.2]{qw:cyclotomic}, Proposition~\ref{prop:kummer}(iii) and Remark~\ref{rem:thetabel}, one has $G/I_{\eth_G}(G)\simeq\Z_\ell^{d-1}\rtimes\Z_\ell$, with $d=\dim(H^1(G,\F_\ell))$.
Therefore, {$I_{\eth_G}(G)=K_{\eth_G}(G)$} is a subgroup of $G$ of infinite index, and thus it is a free pro-$\ell$ group
 by \cite[\S~I.4.5, Exercise~5(b)]{serre:gal}. 
\end{proof}

As mentioned above, Demushkin groups have a single defining relation.
One may prove the Bogomolov-Positselski property also for $1$-relator pro-$\ell$ groups $G$ with quadratic $\F_\ell$-cohomology which can be completed into a Kummerian oriented pro-$\ell$ group $(G,\mathbf{1})$ with a trivial orientation.

\begin{prop}\label{prop:onerel}
 Let $G$ be a finitely generated pro-$\ell$ group with a single defining relation such that 
 \begin{itemize}
  \item[(i)] $H^\bullet(G,\F_\ell)$ is a quadratic algebra;
\item[(ii)] $(G,\mathbf{1})$ is Kummerian.
 \end{itemize}
Then $(G,\bone)$ has the Bogomolov-Positselski property.
\end{prop}

\begin{proof}
 Since $(G,\bone)$ is Kummerian, the quotient $G^{\ab}$ is a free abelian pro-$\ell$ group (cf. Example~\ref{exam:kummer groups}(c)).
We need to show that $G^\prime=K_{\mathbf{1}}(G)=I_{\mathbf{1}}(G)$ is a free pro-$\ell$ group.

Since $G$ has a single defining relation, $H^2(G,\F_\ell)\simeq\F_\ell$ (cf. \cite[Cor.~3.9.5]{nsw:cohn}).
Moreover, since $H^\bullet(G,\F_\ell)$ is quadratic, $H^2(G,\F_\ell)$ is generated by cup products $\chi\cup\psi$ with $\chi,\psi\in H^1(G,\F_\ell)$, so that the cup product from $H^1(G,\F_\ell)$ to $H^2(G,\F_\ell)$ is not trivial (see also \cite[Prop.~4.2]{cq:onerel}).
Consequently, \cite[Cor.~2]{wurfel} yields a short exact sequence of pro-$\ell$ groups
\[
 \xymatrix{\{1\}\ar[r] & N\ar[r] & G\ar[r] &\bar G\ar[r] &\{1\}}
\]
which satisfies the following three properties: $N$ is a free pro-$\ell$ group; $\bar G$ is a Demushkin group; and for every subgroup $S$ of $G$ containing $N$, the inflation map
\begin{equation}\label{eq:wurfel}
 \mathrm{inf}_{S,N}^2\colon H^2(S/N,\F_\ell)\longrightarrow H^2(S,\F_\ell)
\end{equation}
is an isomorphism (this last property is shown to hold in the proof of \cite[Cor.~2]{wurfel}).

Since $G$ is finitely generated, also $\bar G$ is finitely generated. 
Moreover, by \eqref{eq:wurfel} the inflation map $H^2(\bar G,\F_\ell)\to H^2(G,\F_\ell)$ is an isomorphism, and thus by the five-terms exact sequence (cf. \cite[Prop.~1.6.7]{nsw:cohn}) the restriction map 
$$\mathrm{res}_{G,N}^1\colon H^1(G,\F_\ell)\longrightarrow H^1(N,\F_\ell)^G$$ is surjective.
Since $(G,\mathbf{1})$ is Kummerian, and since $\mathrm{res}_{G,N}^1$ is surjective, \cite[Thm.~1.2]{cq:1smooth} implies that also the oriented pro-$\ell$ group $(\bar G,\mathbf{1})=(G,\mathbf{1})/N$ is Kummerian.
Hence, the canonical orientation $\eth_{\bar G}\colon \bar G\to\Z_\ell^\times$ must coincide with the trivial orientation $\bone$ (cf. \cite[Proposition~5.2]{qw:cyclotomic}).
By Theorem~\ref{thm:demushkin}, the oriented pro-$\ell$ group $(\bar G,\mathbf{1})$ has the Bogomolov-Positselski property, and thus $K_{\bone}(\baG)$ --- which coincides with $\baG^\prime$ --- is a free pro-$\ell$ group.

Let $S$ be the normal subgroup of $G$ containing $N$ such that $S/N\simeq \bar G^\prime$.
Thus, $G/S\simeq \bar G/\bar G^\prime$ is abelian, and therefore $S\supseteq G^\prime$.
By \eqref{eq:wurfel}, one has $H^2(S/N,\F_\ell)\simeq H^2(S,\F_\ell)$, and the term on the left-hand side is trivial as $S/N$ is a free pro-$\ell$ group.
Hence, also $H^2(S,\F_\ell)=0$, and $S$ is a free pro-$\ell$ group (cf. \cite[Prop.~3.5.17]{nsw:cohn}).
Since $G^\prime\subseteq S$, and $\cd_\ell(G^\prime)\leq \cd_\ell(S)=1$, $G^\prime$ must be free 
(cf. \cite[\S~3.3, Proposition 14]{serre:gal}).
\end{proof}

\begin{rem}\rm
 Let $F$ be a finitely generated free pro-$\ell$ group, let $r$ be an element of $\Phi(F)$ and let $R$ denote the normal subgroup of $F$ generated by $r$.
Suppose that $\ell\neq2$. By \cite[Prop.~4.2]{cq:onerel} and Example~\ref{exam:kummer groups}(c), the pro-$\ell$ group $G=F/R$ satisfies the conditions (i)--(ii) in Proposition~\ref{prop:onerel} if, and only if, $r\in F'$ and $r\notin F^p\cdot[F',F]$.
\end{rem}

\subsection{Free constructions}
\label{ss:free}
By \cite[\S~3]{efrat:small}, the {\sl free product} of two oriented pro-$\ell$ groups $(G_1,\theta_1)$ and $(G_2,\theta_2)$ is the oriented pro-$\ell$ group $(G,\theta)$ where $G$ is the free pro-$\ell$ product of $G_1,G_2$, and $\theta\colon G\to\Z_\ell^\times$ is the orientation induced by $\theta_1,\theta_2$ via the universal property of $G$ (see also \cite[\S~3.4]{qw:cyclotomic}).

One may extend the above definition to {\sl free amalgamated pro-$\ell$ products} of oriented pro-$\ell$ groups (we refer to \cite[\S~9.2]{ribzal:book} for the definition of free amalgamated pro-$\ell$ products).

\begin{defin}\label{defi:amalg}\rm
 Let $(G_1,\theta_1)$ and $(G_2,\theta_2)$ be two oriented pro-$\ell$ groups such that $G_1$ and $G_2$ have a common subgroup $H\subseteq G_1,G_2$ satisfying $\theta_1\vert_H=\theta_2\vert_H$.
The {\sl amalgamated pro-$\ell$ product of oriented pro-$\ell$ groups} of $(G_1,\theta_1)$ and $(G_2,\theta)$ with amalgamation
in $H$ is the 
oriented pro-$\ell$ group
$(G,\theta)=(G_1,\theta_1)\amalg_H^{\hat{\ell}}(G_2,\theta_2)$, where
$G=G_1\amalg_H^{\hat{\ell}} G_2$ is the free amalgamated pro-$\ell$ product of $G_1$ and $G_2$ over $H$,
and $\theta\colon G\to\Z_\ell^\times$ is the orientation which makes the diagram
\[ \xymatrix@R=0.7truecm{ H\ar[rr]\ar[d] && G_1\ar[d]_-{\varphi_1}\ar@/^1pc/[rdd]^-{\theta_1} & \\
G_2\ar[rr]^-{\varphi_2}\ar@/_1pc/[rrrd]_-{\theta_2} && G\ar@{-->}[rd]|\theta & \\ & && \Z_\ell^\times
}\]
commute.
\end{defin}

Note that the morphisms $\varphi_1$ and $\varphi_2$ may not be injective  (cf. \cite[p.~369]{ribzal:book}).
If they are, the free amalgamated pro-$\ell$ product is said to be {\sl proper}.

If $H=\{1\}$, then $(G_1,\theta_1)\amalg_H^{\hat \ell}(G_2,\theta_2)$ coincides with the free product of oriented pro-$\ell$ groups. In this case we {simply} write $(G_1,\theta_1)\amalg^{\hat \ell}(G_2,\theta_2)$, instead of $(G_1,\theta_1)\amalg_{\{1\}}^{\hat \ell}(G_2,\theta_2)$.
Free products of oriented pro-$\ell$ groups preserve Kummerianity (cf. \cite[Prop.~7.5]{eq:kummer}).

\begin{prop}\label{prop:freeprod kummer}
 Let $(G_1,\theta_1)$ and $(G_2,\theta_2)$ be two Kummerian oriented pro-$\ell$ groups.
Then the free product  $(G_1,\theta_1)\amalg^{\hat \ell}(G_2,\theta_2)$ is again Kummerian.
\end{prop}

We prove that --- under certain conditions --- if the free amalgamated pro-$\ell$ product of two Kummerian oriented 
pro-$\ell$ groups with the Bogomolov-Positselski property is again Kummerian, then it has also the Bogomolov-Positselski property.

\begin{thm}\label{thm:amalg bogo}
 Let $(G_1,\theta_1)$ and $(G_2,\theta_2)$ be torsion free Kummerian oriented pro-$\ell$ groups with the Bogomolov-Positselski property,
with common finitely generated subgroup $U=G_1\cap G_2$ such that $\theta_1\vert_U=\theta_2\vert_U$ and that 
$(U,\theta_U)$ is $\theta_U$-abelian, where $\theta_U=\theta_i\vert_U$ for $i=1,2$.
Suppose that
\begin{itemize}
 \item[(i)] the amalgamated pro-$\ell$ product $(G,\theta)=(G_1,\theta_1)\amalg_U^{\hat \ell}(G_2,\theta_2)$ is Kummerian; 
 \item[(ii)] the restriction maps
$$\mathrm{res}_{G,G_i}^1\colon H^1(G,\F_\ell)\to H^1(G_i,\F_\ell)\qquad\text{and}\qquad\mathrm{res}_{G_i,U}^1\colon H^1(G_i,\F_\ell)\to H^1(U,\F_\ell)$$
are surjective for both $i=1,2$.
\end{itemize}
Then $(G,\theta)$ has the Bogomolov-Positselski property.
\end{thm}

\begin{rem}\label{rem:thm amalg bogo}\rm
 \begin{itemize}
  \item[(a)] If $U$ in the statement of Theorem~\ref{thm:amalg bogo} is the trivial group, then $(G,\theta)$ is the usual free product of  oriented pro-$\ell$ groups, and the two conditions are satisfied by $(G,\theta)$.
For condition~(i), see Proposition~\ref{prop:freeprod kummer}, and condition~(ii) is trivially satisfied.
  Hence, the Bogomolov-Positselski property is preserved by free products of oriented pro-$\ell$ groups.
  
 \item[(b)] By duality, for $i\in\{1,2\}$ the map $\mathrm{res}_{G,G_i}^1$, respectively the map $\mathrm{res}_{G_i,U}^1$, is surjective if, and only if, the map $\bar\iota_{i}\colon G_i/\Phi(G_i)\to G/\Phi(G)$ induced by the inclusion $\iota_{i}\colon U\hookrightarrow G_i$, respectively the map $\bar\iota_{U,i}\colon U/\Phi(U)\to G_i/\Phi(G_i)$ induced by the inclusion $\iota_{U,i}\colon U\hookrightarrow G_i$, is injective.
 
 
 \end{itemize}
\end{rem}

\begin{proof}

By \cite[Thm~A]{cq:bk}, $U$ is a {\sl uniformly powerful pro-$\ell$ group}, and therefore \cite[Prop.~5.22]{qsv:quadratic} implies that $G=G_1\amalg_U^{\hat \ell}G_2$ is a proper amalgam. 
Moreover, by hypothesis one has the monomorphisms of $\ell$-elementary abelian groups $\bar\iota_i$ and $\bar\iota_{U,i}$, with $i=1,2$ (cf. Remark~\ref{rem:thm amalg bogo}(b)).
Hence, also $\bar\iota_U=\bar\iota_i\circ\bar\iota_{U,i}\colon U/\Phi(U)\to G/\Phi(G)$ is injective for both $i=1,2$.

Let $\iota_U\colon U\hookrightarrow G$ be the inclusion of $U$ in $G$, and for $i=1,2$, set 
\[\begin{split}
   \psi_{U}=\pi_{G,\theta}^{\ab}\circ\iota_{U}&\colon U\longrightarrow G(\theta)=G/K_\theta(G),\\
 \psi_i=\pi_{G,\theta}^{\ab}\circ\iota_i&\colon G_i\longrightarrow G(\theta)=G/K_\theta(G).
  \end{split}\]
Then 
\begin{equation}\label{eq:kernels amalg}
 \kernel(\psi_{U})=U\cap K_{\theta}(G)\qquad\text{and}\qquad \kernel(\psi_i)=G_i\cap K_\theta(G).
\end{equation}
Now consider the commutative diagram
\begin{equation}\label{eq:commdiag amalgam}
 \xymatrix@R=0.6truecm{ &&&& \\
    U\ar@{.>}[rrrrd]\ar@{->>}@/^2pc/[rrrrrr]\ar[rr]|-{\iota_{U,1}}\ar[rrdd]|-{\iota_{U,2}}\ar[rrd]|-{\iota_U}
                                        && G_1\ar[d]_-{\iota_1}\ar@{->>}[rr]\ar[drr]|-{\psi_1} && G_1/\Phi(G_1)\ar[rd]|-{\bar\iota_1}
                                        && U/\Phi(U)\ar@/^2pc/[lldd]|-{\bar\iota_{U,2}}\ar@/^/[ll]|-{\bar\iota_{U,1}}
                                        \ar@/^/[dl]|-{\bar\iota_U} \\
   && 
  G\ar@{->>}[rr]|-{\pi_{G,\theta}^{\ab}}  && G(\theta)\ar@{->>}[r]     & G/\Phi(G) & \\
                                        && G_2\ar[u]^-{\iota_2}\ar@{->>}[rr]\ar[urr]|-{\psi_2} && G_2/\Phi(G_2)\ar[ru]|-{\bar\iota_2} & & }
\end{equation} 
where the dotted arrow from $U$ to $G(\theta)$ is $\psi_U$.
By Remark~\ref{rem:pthetabel subgroups}, the oriented pro-$\ell$ groups $(\image(\psi_U),\theta\vert_{\image(\psi_U)})$ and $(\image(\psi_i),\theta\vert_{\image(\psi_i)})$ are  $\theta\vert_{\image(\psi_U)}$- and $\theta\vert_{\image(\psi_i)}$-abelian, {respectively}. In particular, 
\begin{equation}\label{eq:ker psi i}
 \kernel(\psi_i)\supseteq I_{\theta_i}(G_i)=K_{\theta_i}(G_i),
\end{equation}
where the left-hand side inclusion follows by Proposition~\ref{prop:maxTh}, and the right-side equality follows by Proposition~\ref{prop:kummer}(iv), as $(G_i,\theta_i)$ is Kummerian for $i\in\{1,2\}$ by hypothesis.
Consequently, the pro-$\ell$ groups $\image(\psi_U)$ and $\image(\psi_i)$ are torsion-free, so that $\kernel(\psi_U)$ and $\kernel(\psi_i)$ are self-isolated subgroups of $U$ and $G_i$ respectively.
On the other hand, by duality one has $\kernel(\psi_U)\subseteq\Phi(U)$ and $\kernel(\psi_i)\subseteq\Phi(G_i)$, as the maps $\bar\iota_U$ and $\bar\iota_i$ are injective.
Altogether, by \eqref{eq:kernels amalg} and \eqref{eq:ker psi i} one has
$$K_{\theta\vert_U}(U)=\{1\}\subseteq U\cap K_{\theta}(G)\subseteq\Phi({U})\qquad\text{and}\qquad 
K_{\theta_i}(G_i)\subseteq G_i\cap K_{\theta}(G)\subseteq \Phi(G_i),$$
and thus $\{1\}= U\cap K_{\theta}(G)$ and $K_{\theta_i}(G_i)= G_i\cap K_{\theta}(G)$ by Proposition~\ref{prop:fratK}.

Now, let $\mathcal{T}=(\euV(\mathcal{T}),\euE(\mathcal{T}))$ be the 
pro-$\ell$ tree whose vertices and edges are given by 
\[
 \euV(\mathcal{T})=\{\:gG_{1},gG_{2}\:\mid\:g\in G\}\quad\text{and}\quad
 \euE(\mathcal{T})=\{\,gU,\overline{gU}\mid g\in G\,\},
\]
respectively. {In particular, every edge $gU\in\euE(\mathcal{T})$ defines an origin, the $G_1$-coset $gG_1$ 
and a terminus, the $G_2$-coset $gG_2$. For $\overline{gU}\in\euE(\mathcal{T})$ the roles of the terminus and origin
are interchanged.}
Then $\mathcal{T}$ is a second countable pro-$\ell$ tree, with a natural $G$-action (cf. \cite[Example~6.2.3]{ribes:book}).
For $v=gG_i\in\euV(\mathcal{T})$ and $\eue=hU\in\euE(\mathcal{T})$, with $g,h\in G$ and $i\in\{1,2\}$, let $K_v$ and $K_{\eue}$ denote  the stabilizers of $v$ and $\eue$ in $K_\theta(G)$, respectively.
Hence 
\[   \begin{split}
      K_v &=\{\:x\in K_\theta(G)\:\mid\:xg\in g G_i\:\}= K_\theta(G)\cap gG_ig^{-1},\\ 
      K_{\eue}&=\{\:x\in K_\theta(G)\:\mid\:xh\in h U\:\}= K_\theta(G)\cap hUh^{-1}.
     \end{split}\]
Since $K_\theta(G)$ is a normal subgroup of $G$, for every $v=gG_i\in\euV(\mathcal{T})$ the subgroup $K_v$ is isomorphic to $K_\theta(G)\cap G_i=K_{\theta_i}(G_i)$, which is free by hypothesis; while for every $\eue=hU\in\euE(\mathcal{T})$ the subgroup $K_{\eue}$ is equal to $\{1\}$, and hence no non-trivial element of $K_\theta(G)$ stabilizes an edge.
Therefore, by \cite[Thm.~5.6]{melnikov:freeprod}, $K_\theta(G)$  has the following decomposition as free pro-$\ell$ product:
\begin{equation}
\label{eq:deco}
 K_\theta(G)=\left(\coprod_{v\in\euV'} K_v\right)\amalg F,
\end{equation}
for some subset $\euV'$ of $\euV(\mathcal{T})$, where $F$ is a free pro-$\ell$ group.
Hence $K_\theta(G)$ is the free pro-$\ell$ product of free pro-$\ell$ groups, and thus it is a free pro-$\ell$ group as well.
\end{proof}



\begin{exam}\label{ex:square diag}\rm
 Let $(G_1,\theta_1)$ and $(G_2,\theta_2)$ be the oriented pro-$\ell$ groups with 
 \[\begin{split}
    G_1&=\left\langle\: x, y_1,y_3\:\mid\:[y_1,y_3]=1,\:{}^xy_j=y_j^{1+\ell},\;\forall\:j\in\{1,3\}\:\right\rangle
    \simeq\Z_\ell^2\rtimes\Z_\ell, \\
    G_2&=\left\langle\: x, y_2,y_3\:\mid\:[y_2,y_3]=1,\:{}^xy_j=y_j^{1+\ell},\;\forall\:j\in\{2,3\}\:\right\rangle
    \simeq\Z_\ell^2\rtimes\Z_\ell,
   \end{split}\]
   and such that $\theta_i(x)=1+\ell$ and $\theta_i(y_i)=\theta_i(y_3)=1$ for both $i=1,2$.
   By Remark~\ref{rem:thetabel}, these two oriented pro-$\ell$ groups are respectively $\theta_1$- and $\theta_2$-abelian.
Set $U=G_1\cap G_2$ --- i.e. $U$ is the subgroup generated by $x,y_3$. 
Clearly, $\theta_1\vert_U=\theta_2\vert_U$, and 
$$(U,\theta_i\vert_U)=\langle\:y_3\:\rangle\rtimes(\langle\:x\:\rangle,\theta_i\vert_{\langle\:x\:\rangle})
\qquad\text{for both }i=1,2,$$
which is $\theta_i\vert_U$-abelian by Remark~\ref{rem:pthetabel subgroups}.
Moreover, it is straightforward to see that the maps $\bar\iota_{U,i}\colon U/\Phi(U)\to G_i/\Phi(G_i)$ are injective for both $i=1,2$.
Now let $(G,\theta)$ be the oriented pro-$\ell$ group $(G_1,\theta_1)\amalg_U^{\hat\ell}(G_2,\theta_2)$. Then
\[
  G=\left\langle\: x, y_1,y_2,y_3\:\mid\:[y_1,y_3]=[y_2,y_3]=1,\:{}^xy_i=y_i^{1+\ell},\;\forall\:i\in\{1,2,3\}\:\right\rangle
 \]
and $\theta(x)=1+\ell$, $\theta(y_j)=1$ for $j=1,2,3$.
Moreover, one has an epimorphism of oriented pro-$\ell$ groups $\tau\colon(G,\theta)\to(\bar G,\bar\theta)$, where
\[
  \bar G=\left\langle\: \bar x, \bar y_1,\bar y_2,\bar y_3\:\mid\:[\bar y_j,\bar y_{j'}]=1,\:{}^{\bar x}\bar y_j=\bar y_j^{1+\ell},\;\forall\:j,j'\in\{1,2,3\}\:\right\rangle\simeq\Z_\ell^3\rtimes\Z_\ell,
 \]
and $\bar x=\tau(x)$, $\bar y_j=\tau(y_j)$ for $j=1,2,3$. 
By Remark~\ref{rem:thetabel}, $(\bar G,\bar\theta)$ is $\bar\theta$-abelian, and thus $\kernel(\tau)\supseteq I_\theta(G)$ by Proposition~\ref{prop:maxTh}.
On the other hand, it is straightforward to see that $\Phi(G)\supseteq\kernel(\tau)$, and hence $(G,\theta)$ is Kummerian by Proposition~\ref{prop:kummer}--(vi).
Since $(G_1,\theta_1)$ and $(G_2,\theta_2)$ have the Bogomolov-Positselski property by Example~\ref{exam:bogo}--(a), Theorem~\ref{thm:amalg bogo} implis that also $(G,\theta)$ has the Bogomolov-Positselski property.
Observe that $G$ is $H^\bullet$-quadratic (cf. \cite[Rem.~5.25--(c)]{qsv:quadratic}).
\end{exam}



\subsection{Pro-$\ell$ groups of elementary type}
\label{ssec:ETC}

Let $ (G,\theta)$ be an oriented pro-$\ell$ group, and let $A$ be a free abelian pro-$\ell$ group.
Recall that the {\sl semidirect product} $A\rtimes(G,\theta)=(A\rtimes G,\theta\circ\pi)$ is the oriented pro-$\ell$ group where $gag^{-1}=a^{\theta(g)}$ for all $a\in A$ and $g\in G$, and $\pi\colon A\rtimes G\to G$ is the canonical projection (cf. \cite[\S~3]{efrat:small}).

The following is straightforward (cf., e.g., \cite[Prop.~3.6]{eq:kummer}).

\begin{prop}\label{prop:semidirect kummer}
Given an oriented pro-$\ell$ group $ (G,\theta)$ and a free abelian pro-$\ell$ group $A$, one has 
$K_{\theta\circ\pi}(A\rtimes G)=K_\theta(G)$.
In particular, $A\rtimes(G,\theta)$ is Kummerian if, and only if, $(G,\theta)$ is Kummerian; and 
$A\rtimes(G,\theta)$ has the Bogomolov-Positselski property if, and only if, $(G,\theta)$ has the Bogomolov-Positselski property.
\end{prop}

The family $\ET_\ell$ of {\sl oriented pro-$\ell$ groups of elementary type} is the smallest class 
of finitely generated oriented pro-$\ell$ groups satisfying (cf. \cite[\S~3]{efrat:small})
\begin{itemize}
 \item[(a)] the oriented pro-$\ell$ group $(G,\eth_G)$, with $G$ a Demushkin group, is of elementary type;
 \item[(b)] the oriented pro-$\ell$ group $ (\Z_\ell,\theta)$, with $\theta\colon \Z_\ell\to\Z_\ell^\times$ arbitrary, is of elementary type;
 \item[(c)] if the oriented pro-$\ell$ group $ (G,\theta)$ is of elementary type and $A$ is a finitely generated free abelian pro-$\ell$ group,
then also the semidirect product $A\rtimes(G,\theta)$ is of elementary type;
 \item[(d)] if $(G_1,\theta_1)$ and $(G_2,\theta_2)$ are oriented pro-$\ell$ groups of elementary type then also the free pro-$\ell$ product $(G_1,\theta_1)\amalg^{\hat \ell}(G_2,\theta_2)$ is of elementary type.
\end{itemize}

\begin{rem}\label{rem:ETC}\rm
 \begin{itemize}
  \item[(a)] In the original definition of oriented pro-2 groups of elementary type one has that also the cyclic group $C_2$ of order 2, endowed with the non-trivial orientation $\theta_{C_2}\colon C_2\twoheadrightarrow\{\pm1\}\subset\Z_2^\times$, is a pro-2 group of elementary type (cf. \cite[p.~242]{efrat:small}).
  Since our results always assume oriented pro-$\ell$ groups to be torsion-free, we may safely exclude $(C_2,\theta_{C_2})$ from the above definition of oriented pro-$\ell$ groups of elementary type.
  \item[(b)] From the results in \cite[\S~3.3--3.4]{qw:cyclotomic}, one may deduce that a finitley generated subgroup $H$ of an oriented pro-$\ell$ groups of elementary type $(G,\theta)$ gives rise to a pro-$\ell$ groups of elementary type $(H,\theta\vert_H)$.
  \item[(c)] If $(F,\theta)$ is a torsion-free oriented pro-$\ell$ group with $F$ a finitely generated free pro-$\ell$ group and $\theta\colon F\to\Z_\ell^\times$ any orientation, then $(F,\theta)$ is of elementary type.
  Indeed, if $\theta=\bone$, then $(F,\theta)$ is isomorphic to the free pro-$\ell$ product of $d$ copies of the oriented pro-$\ell$ group $(\Z_\ell,\bone)$, where $d$ is the minimal number of generators of $F$.
  Otherwise, $\image(\theta)\simeq\Z_\ell$, and the short exact sequence of pro-$\ell$ groups
  \[
   \xymatrix{ \{1\}\ar[r] & \kernel(\theta)\ar[r] & F\ar[r]& \image(\theta)\ar[r]&\{1\} }
  \]
splits.
In this case, let $\{x_1,\ldots,x_d\}$ be a minimal generating set where $\theta(x_1)\neq1$ and $\theta(x_i)=1$ for $i\geq2$, and let $H$ be the subgroup of $F$ generated by $\{x_2,\ldots,x_d\}$, which is free. 
Then, 
$(F,\theta)\simeq (H,\bone)\amalg^{\hat\ell}(\image(\theta),\mathrm{id}_{\image(\theta)})$, where both factors are oriented pro-$\ell$ groups of elementary type.
 \end{itemize}
 \end{rem}

From Example~\ref{exam:kummer groups}--(b), \S~\ref{ssec:demushkin}, and Propositions \ref{prop:freeprod kummer} and~\ref{prop:semidirect kummer},
one concludes that oriented pro-$\ell$ groups of elementary type are Kummerian.
I.~Efrat's {\sl Elementary Type Conjecture} states that if $\K$ is a field containing a primitive $\ell^{\mathrm{th}}$-root of 1 (and also $\sqrt{-1}$ if $\ell=2$)
and if the maximal pro-$\ell$ Galois group $G_{\K}(\ell)$ is finitely generated, then 
$(G_{\K}(\ell),\ttheta_{\K,\ell})$ is of elementary type (cf. \cites{efrat:etc1,efrat:etc2}, see also \cite[\S~10]{marshall:etc} and \cite[\S~7.5]{qw:cyclotomic}).

\begin{exam}\rm
 The oriented pro-$\ell$ group $(G,\theta)$ as in Example~\ref{ex:square diag} is not of elementary type.
 Indeed, the subgroup of $G$ generated by $\{x,y_1,y_2\}$ contains a finitely generated subgroup which does not complete into a Kummerian oriented pro-$\ell$ group (cf. \cite[Ex.~5.3]{cq:1smooth}) --- in particular, $G$ does not occur as the maximal pro-$\ell$ Galois group of a field containing a primitive $\ell^{\mathrm{th}}$-root of unity (and also $\sqrt{-1}$ if $\ell=2$).
Therefore, $(G,\theta)$ is not of elementary type by Remark~\ref{rem:ETC}--(b).
\end{exam}

\begin{thm}\label{cor:etc}
 Let $ (G,\theta)$ be an oriented pro-$\ell$ group of elementary type.
Then $(G,\theta)$ has the Bogomolov-Positselski property.
\end{thm}

\begin{proof}
If $G$ is a free pro-$\ell$ group, then $(G,\theta)$ has the Bogomolov-Positselski property by Example~\ref{exam:bogo}--(a).
If $G$ is a Demushkin group and $\theta=\eth_G$, then $(G,\eth_G)$ has the Bogomolov-Positselski property by Theorem~\ref{thm:demushkin}.

By Proposition~\ref{prop:semidirect kummer}, if $ (G,\theta)=A\rtimes(G_0,\theta\vert_{G_0})$ where $A$ is a free abelian pro-$\ell$ group and the right side factor is an oriented pro-$\ell$ group of elementary type, then $(G,\theta)$ has the Bogomolov-Positselski property --- provided that $(G_0,\theta\vert_{G_0})$ has the Bogomolov-Positselski property.

Finally, by Theorem~\ref{thm:amalg bogo}, if $(G,\theta)=(G_1,\theta_1)\amalg^{\hat \ell}(G_2,\theta_2)$ and both $(G_1,\theta_1)$ and $(G_2,\theta_2)$ have the Bogomolov-Positselski property, then also $(G,\theta)$ has the Bogomolov-Positselski property.
\end{proof}

Let $\K$ be a field containing a primitive $\ell^{\mathrm{th}}$-root of unity, and set $\K^\times=\K\smallsetminus\{0\}$.
Since Kummer theory yields an isomorphism of (discrete) $\ell$-elementary abelian pro-$\ell$ groups $H^1(G_{\K}(\ell),\F_\ell)^\vee\simeq \K^\times/(\K^\times)^\ell$, the pro-$\ell$ group $G_{\K}(\ell)$ is finitely generated if, and only if, the quotient $\K^\times/(\K^\times)^\ell$ is finite.
One has the following (see \cite[Thm.~D]{MPQT}, and \cite{efrat:fingen} for item~(f)).

\begin{prop}\label{prop:fields ET}
 Let $\K$ be a field containing a primitive $\ell^{\mathrm{th}}$-root of 1 \textup{(}and also $\sqrt{-1}$ if $\ell=2$\textup{)}, 
 such that the quotient $\K^\times/(\K^\times)^\ell$ is finite. 
 Then the oriented pro-$\ell$ group $(G_{\K}(\ell),\theta_{\K,\ell})$ is of elementary type in the following cases:
 \begin{itemize}
  \item[(a)] $\K$ is finite;
\item[(b)] $\K$ is a pseudo algebraically closed (PAC) field, or an extension of relative trascendence degree 1 of a PAC field;
\item[(c)] $\K$ is an extension of trascendence degree 1 of a local field;
\item[(d)] $\K$ is $\ell$-rigid {\rm (}cf. \cite[p.~722]{ware}, see also \cite[\S~3]{cmq:fast}{\rm )};
\item[(e)] $\K$ is algebraic extension of a global field of characteristic not $\ell$;
\item[(f)] $\K=\Bbbk(\!(T)\!)$, where $(G_{\Bbbk}(\ell),\theta_{\Bbbk,\ell})$ is of elementary type.
 \end{itemize}
\end{prop}

Corollary~\ref{cor:fields} follows from Theorem~\ref{cor:etc} and Proposition~\ref{prop:fields ET}.



\end{document}